\RequirePackage{fix-cm}
\documentclass{svjour3-noname}
\smartqed  
\usepackage{graphicx}
\usepackage{graphicx}
\usepackage{multirow}
\usepackage{amsmath,amssymb,amsfonts}
\usepackage{bm}
\usepackage{natbib}
\usepackage{booktabs}

\usepackage{tikz}
\usetikzlibrary{positioning,shapes.multipart}
\usepackage{forest}

\tikzset{
  font=\normalsize,
  tree node/.style = {align=center, inner sep=0pt, draw, circle, minimum size=18},
  tree node label/.style={font=\scriptsize},
}

\forestset{
  declare toks={left branch prefix}{},
  declare toks={right branch prefix}{},
  declare toks={left branch suffix}{},
  declare toks={right branch suffix}{},
  maths branch labels/.style={
    branch label/.style={
      if n=1{
        edge label={node [left, midway] {$\forestoption{left branch prefix}##1\forestoption{left branch suffix}$}},
      }{
        edge label={node [right, midway] {$\forestoption{right branch prefix}##1\forestoption{right branch suffix}$}},
      }
    },
  },
  set branch labels/.style n args=4{%
    left branch prefix={#1},
    left branch suffix={#2},
    right branch prefix={#3},
    right branch suffix={#4},
  },
  branch and bound/.style={
    /tikz/every label/.append style=tree node label,
    maths branch labels,
    for tree={
      tree node,
      math content,
      s sep'+=20mm,
      l sep'+=5mm,
      thick,
      edge+={thick},
    },
    before typesetting nodes={
      for tree={
        split option={content}{:}{content,branch label},
      },
    },
  },
}

\usepackage[linesnumbered,ruled]{algorithm2e}
\usepackage{subcaption}

\newcommand{\Expect}{\mathbb{E}}
\newcommand{\tp}{\mathsf{T}}

\newcommand{\bfeta}{\bm{\eta}}
\newcommand{\bfxi}{\bm{\xi}}
\newcommand{\bfd}{\mathbf{d}}

\newcommand{\bfv}{\mathbf{v}}

\newcommand{\xd}{n}
\newcommand{\etad}{\xd}
\newcommand{\xid}{s}
\newcommand{\yd}{p}
\newcommand{\zd}{q}
\newcommand{\rveta}{\bm{\eta}} 
\newcommand{\seta}{\eta} 
\newcommand{\rvxi}{\bm{\xi}}
\newcommand{\sxi}{\xi}
\newcommand{\Proj}{\operatorname{Proj}}

\newcommand{\Eta}{\mathrm{H}}

\newcommand{\zpesp}{z_{\mbox{\sc pesp}}}

\journalname{Mathematical Programming}

\begin{document}

\title{Probing-Enhanced Stochastic Programming
}


\author{Zhichao Ma        \and
        Youngdae Kim \and Jeff Linderoth \and James R. Luedtke \and Logan R. Matthews
}


\institute{
Zhichao Ma \and Jeff Linderoth \and Jim Luedtke \at Department of Industrial and Systems Engineering, 
              University of Wisconsin-Madison 
              \email{zma59@wisc.edu, linderoth@wisc.edu, jim.luedtke@wisc.edu}
\and Jeff Linderoth \and Jim Luedtke \at  Wisconsin Institute for Discovery, University of Wisconsin-Madison
\and Youngdae Kim \and Logan R. Matthews \at Energy Sciences, ExxonMobil Technology and Engineering Company
}

\date{Received: June, 2024 / Accepted: date}

\maketitle

\begin{abstract}
We consider a two-stage stochastic decision problem where the decision-maker has the opportunity to obtain information about the distribution of the random variables $\bfxi$ that appear in the problem through a set of discrete actions that we refer to as \emph{probing}.  Probing components of a random vector $\bfeta$ that is jointly-distributed with $\rvxi$ allows the decision-maker to learn about the conditional distribution of $\rvxi$ given the observed components of $\rveta$.
We propose a three-stage optimization model for this problem, where in the first stage some components of $\rveta$ are chosen to be observed, and decisions in subsequent stages must be consistent with the obtained information.  In the case that $\rveta$ and $\rvxi$ have finite support, Goel and Grossmann gave a mixed-integer programming (MIP) formulation of this problem whose size is proportional to the square of cardinality of the sample space of the random variables.  
We propose to solve the model using bounds obtained from an information-based relaxation, combined with a branching scheme that enforces the consistency of decisions with observed information.  The branch-and-bound approach can naturally be combined with sampling in order to estimate both lower and upper bounds on the optimal solution value and does not require $\rveta$ or $\rvxi$ to have finite support. We conduct a computational study of our method on instances of a stochastic facility location and sizing problem with the option to probe customers to learn about their demands before building facilities. We find that on instances with finite support, our approach scales significantly better than the MIP formulation and also demonstrate that our method can compute statistical bounds on instances with continuous distributions that improve upon the perfect information bounds. 
\end{abstract}

\section{Introduction}
\label{sec:introduction}

We study optimization models designed to understand the value of obtaining information about the outcome of random variables $\rvxi$.  The models are stochastic optimization problems with a specific but important form of decision-dependent uncertainty.  
In stochastic programming, the literature on decision-dependent uncertainty typically divides models into two primary categories \citep{Tarhan2009,hellemo.barton.tomasgard:18}.  In the first category, the probability distribution of $\rvxi$ is a function of the decision variables, and models in this class are said to have \emph{decision-dependent probabilities}.  In the second category, the timing of the revelation of the outcome of $\rvxi$ is decision-dependent, and models in this class are said to have \emph{decision-dependent information structure}.

Models with decision-dependent information structure are most applicable in a multi-stage setting, where a sequence of decisions are made in stages with the opportunity to observe random outcomes between stages.  As pointed out by \citet{dupacova:06}, it is precisely in these contexts where their solution is most challenging.  First-stage decisions may influence the marginal and conditional probability distributions in subsequent stages, or if the decisions influence the time when uncertainty is resolved, then the nonanticipativity of future decisions must depend on decisions made at the current stage.

Most of the research on computational approaches for stochastic programs under decision-dependent uncertainty work with a nonanticipative formulation in a setting where the uncertainty is modeled by a finite set of scenarios. The key to this approach is that the collection of nonanticipativity conditions to be enforced depends on the decisions.  \citet{goel.grossmann:04,Goel2006} present a MIP (disjunctive) model for this problem, using binary variables to model the time at
which each endogenous uncertain parameter is observed.  However, the number of (conditional) non-anticipativity constraints form a very large class, roughly on the order of the number of scenarios squared.

Significant subsequent work has been based on improving computational performance of the nonanticapative model by techniques such as Lagrangian relaxation \citep{Goel2006,Tarhan2009}, branch-and-cut \citep{colvin.maravelias:09,colvin.maravelias:10}, and redundant constraint identification \citep{boland.et.al:16}.
A different approach to addressing decision-dependent uncertainty, using linear and piecewise-linear decision rules was given by \citet{vayanos:11}. \citet{Mercier2008} propose a method that is applicable for multi-stage problems with decision-dependent information structure that first applies a sample average approximation (SAA) and then converts that problem to a standard Markov decision process and solves it. \cite{SOLAK2010} use an SAA method to solve an R\&D portfolio management problem, and solve the SAA problem with a Lagrangian relaxation approach.



In our work, we consider a two-stage stochastic programming model as our starting point, where some decisions $y$ are made before the outcome of a random variable $\rvxi$ is revealed and recourse decisions $z(y,\sxi)$ are made after observing $\sxi$.  We extend the two-stage model to allow the decision maker to perform a set of discrete actions, which we refer to as \emph{probing}, to reveal information about the distribution of $\rvxi$, before making their decisions $y$.  We call this model a \emph{Probing-Enhanced Stochastic Program}.  Different from most existing work on stochastic programming with decision-dependent information structure, we do not necessarily assume that taking a probing action gives us complete information about some of the elements of $\rvxi$.  Rather, each candidate probing action gives us the opportunity to observe the value of a component of a different random vector $\rveta$, which we assume is correlated with $\rvxi$, thus giving \emph{information} about the distribution of $\rvxi$  that may help the decision-maker make a better decision.  Our work is related to the discussion of using $\sigma$-fields to model the evolution of information described in \citet{artstein:99} and \citet{artstein.wets:93}.  

\paragraph{Contributions: }
A primary contribution of our work is to model the problem using a \emph{nested} stochastic programming formulation of the uncertainty, rather than a nonanticiapative one.  This allows us to forgo the modeling of conditional non-anticipativity constraints and allows observation (or sampling) of random variables conditional on the outcomes observed by probing.  The flexibility afforded by a nested formulation allows for our model to handle random variables $\rveta$ arising from a continuous probability distribution.

A second contribution of our work is a branch-and-bound algorithm to solve the nested model that relies on information-based relaxations to generate bounds at nodes of the branch-and-bound tree. The method is exact when the support of the random variables is small enough to allow exact calculation of conditional expectations, and we show how statistical estimates of bounds on the optimal solution value of the nested formulation can be obtained via both internal and external sampling procedures. The external sampling procedure is related to the sample average approximation approach used by \citet{SOLAK2010}, but we use a different method to solve the SAA problem. The internal sampling method can be interpreted as an extension of the stochastic branch-and-bound method of \cite{norkin1998branch} to our probing-enhanced model.  

We also propose a greedy heuristic for generating quality solutions that efficiently re-uses computations in order to approximate solution quality when comparing potential candidate solutions within the heuristic.  

We give computational results that demonstrate that our model can be orders of magnitude faster for the exact solution of small instances with finite support of the random outcomes, when compared to the nonanticaptive formulation, and can improve significantly over the perfect information bound for large instances and instances with continuous distribution.   To our knowledge, this is the first computational procedure that is able to significantly improve over perfect-information-based bounds for large-scale problem instances in this class.

\paragraph{Contents: } The remainder of the paper is organized as follows.  
In Section~\ref{sec:Model}, we describe our modeling framework for probing-enhanced stochastic programming, giving both nested and non-anticipative formulations.  The nested formulation is posed as a combinatorial optimization problem that selects the set of probing decisions that maximize the expected value, and in Section~\ref{sec:Branch_and_Bound}, we describe a branch-and-bound method for its solution that uses information-relaxation based bounds, as well as two different branching methods.
We describe in Section~\ref{sec:Sampling} two sampling methodologies that may be combined with our branch-and-bound method in both an external and internal manner to obtain statistical upper bounds for the case when the expected value calculations of the stochastic program cannot be carried out exactly. We also describe how sampling can be used with a candidate probing decision to estimate a statistical lower bound. We describe our greedy heuristic in Section~\ref{sec:heuristic} and present results of computational experiments in Section~\ref{sec:comp}.

\paragraph{Notation: }
We use bold-face math symbols, such as $\bfxi$, to denote random variables, and light-typeface symbols, such as $\bfxi(\omega) = \sxi$ to denote realizations of random variables.  For a random variable $\bfxi$, the notation $\mathbb{E}_{\bfxi}[\cdot]$ denotes expectation of the argument with respect to the probability measure associated with the random variable $\bfxi$.  The set of integers $\{1,2,\ldots,n\}$ is denoted as $[n]$. 

\section{Background and Models} \label{sec:Model}



Our starting point is a classic two-stage stochastic programming problem
\begin{equation}
\label{eq:2sp}
    \max_{y \in Y} \beta^\tp y + \Expect_{\rvxi}[ Q(y, \rvxi) ],
\end{equation}
where $\rvxi: \Omega \rightarrow \mathbb{R}^\xid$ is a $\xid$-dimensional random vector defined on a probability space $(\Omega,\cal F, \mathbb{P})$ having support $\Xi$, and
\begin{equation*}
    Y := \{y \in \mathbb{R}^{\hat{\yd}} \times \mathbb{Z}^{\yd-\hat{\yd}}_+ \ : \ Ay = b\}
\end{equation*}
is a mixed-integer set in $\mathbb{R}^\yd$.  Given a decision vector $y \in Y$ and outcome $\rvxi(\omega) = \sxi \in \Xi$, the function
\begin{equation}
\label{eq:recourse-function}
    Q(y,  \sxi) = \max \{ \gamma(\xi)^\tp z : z \in Z(y,\sxi) \}
\end{equation}
optimizes the recourse decision $z$, where 
\begin{equation*}
    Z(y,\xi) := \left\{ z \in \mathbb{R}^{\hat{\zd}}_+ \times \mathbb{Z}^{\zd-\hat{\zd}}_+ : \ T(\sxi) y + W(\sxi) z = h(\sxi) \right\}
\end{equation*}
is a mixed-integer set in $\mathbb{R}^\zd$  for all  $y \in Y$ and $\xi \in \Xi$. 
We assume that $Q(y, \sxi)$ is finite for all $y \in Y$ and $\sxi \in \Xi$, which is the standard assumption of \emph{relatively complete recourse} in stochastic programming.

We wish to model a situation in which information about the distribution of $\rvxi$ can be obtained by taking some discrete actions, which we refer to as \emph{probing}. Abstractly, we model this by considering an $\etad$-dimensional random vector $\rveta: \Omega \rightarrow \mathbb{R}^n$ with support $\Eta$ such that $(\rveta,\rvxi)$ is a jointly-distributed random vector defined on the same probability space $(\Omega, \cal F, \mathbb{P})$.  

As a special case of our modeling framework, we may set $\rveta \equiv \rvxi$, in which case the probing decisions correspond to observing individual elements of the random vector $\rvxi$ in \eqref{eq:2sp}. We will use this special case in our numerical experiments, but the methodology described in the paper applies to the general model, with the main computational requirement that samples of $\rvxi$ conditional on having observed a subset of the entries of $\rveta$ can be generated.

We also introduce a \emph{cost-to-probe function} $\alpha: 2^{[n]} \rightarrow \mathbb{R}_+$, where $\alpha(S)$ for $S \subseteq[n]$ is the cost of probing the $S$ components of the random vector $\rveta$, which we denote as $\rveta_S$.  For the most part, we assume only that $\alpha(\cdot)$ is a monotone-increasing function; that is, if $S \subseteq T$, then $\alpha(S) \leq \alpha(T)$. One of our branching methods requires the further assumption that $\alpha(\cdot)$ is modular, i.e., $\alpha(S) = \sum_{i \in S} \alpha(\{i\})$ for all subsets $S$.

\subsection{Probing-Enhanced Stochastic Program}
\label{sec:pesp}

In our model, the decision-maker observes the realizations $\rveta_S$ before choosing $y$, so these decisions $y$ can depend on 
the observed values $\seta_S$.  This can be advantageous when $\rvxi$ and $\rveta$ are correlated, as the first-stage decisions $y$ may be optimized with respect to the conditional distribution of $\rvxi$ given the observations $\seta_{S}$.

Define the \emph{conditional-recourse function} analog of \eqref{eq:recourse-function} given the observations $\rveta_S = \seta_S$ for $S \subseteq [n]$ as follows:
\begin{equation}
\label{eq:conditional-recourse}
 R(\seta_S) = \max_{y \in Y} \beta^\tp y + \Expect_{\rvxi} \Bigl[ Q(y,\rvxi)|\rveta_S = \eta_S \Bigr],
\end{equation}
and for each subset $S \subseteq [n]$ of probing decision, define the \emph{expected conditional-recourse function}
\begin{equation*}
F(S) := \Expect_{\rveta_S} [R(\rveta_{S})].
\end{equation*}
Thus, when we introduce the option of probing, the two-stage stochastic programming problem~\eqref{eq:2sp} is extended to the \emph{probing-enhanced stochastic program}
\begin{align}
\label{eq:mainprob} \tag{PESP}
    \zpesp = \max_{S \subseteq [n]} F(S) -\alpha(S).
\end{align}
Problem \eqref{eq:mainprob} seeks to find the optimal set of probing decisions that trades-off the probing cost incurred against the potential gains obtained from knowing more about the distribution of the random variables $\rvxi$ when making first-stage decisions $y$.

Calculating the conditional recourse function $R(\eta_S)$ and the expected conditional recourse function $F(S)$ both involve the evaluation of an expectation.
In the case that that the random variables have finite support, i.e., both $\Eta$ and $\Xi$ are both finite sets, we can more explicitly detail the computations required to compute $F(S)$.

First, for each $\seta_S \in \Proj_{S}(\Eta)$, the value of the following finite-support two-stage stochastic program must be calculated:
\begin{equation}
\label{eq:r_etas_s_eval}
R(\seta_S) = \max_{y \in Y} \beta^\tp y + \sum_{\sxi \in \Xi} \mathbb{P}(\rvxi = \sxi | \rveta_S = \eta_S) Q(y,\sxi),
\end{equation}
where $Q(y,\sxi)$ is defined in \eqref{eq:recourse-function}.
If we define 
\begin{equation}
\label{eq:define_xi_s}
\Xi(\seta_S) := \{ \rvxi(\omega) : \omega \in \Omega, \rveta_S(\omega) = \seta_S \} \subseteq \Xi
\end{equation}
as the outcomes of $\rvxi$ that are possible after observing ${\seta}_S$, i.e., the scenarios in the conditional distribution, then the evaluation of $R(\seta_S)$ in~\eqref{eq:r_etas_s_eval} requires the solution of a two-stage stochastic program with $|\Xi(\seta_S)|$ scenarios.  Note that there are $|\Proj_S(H)|$ such two-stage stochastic programs to be solved.
Given these values, the value of $F(S)$ is then simply calculated as
\[ F(S) := \Expect_{\rveta_S} [R(\rveta_{S})] = 
\sum_{\seta_S \in \Proj_S(\Eta)} \mathbb{P}(\rveta_S = \seta_S) R(\seta_S). \]

In our computations, all two-stage stochastic programs are solved by directly formulating the extensive form stochastic program \citep{SPbook_birge1997}.  The evaluation of $F([n])$ requires solving $|\Xi|$ one-scenario problems.  The evaluation of $F(\emptyset)$ requires solving one $|\Xi|$-scenario two-stage stochastic program.  In general, the smaller the set $S$, the more computationally challenging the evaluation of $F(S)$ becomes. 

\subsection{A Non-Anticipativity Based Formulation}

In this section, we describe an explicit deterministic equivalent mixed integer programming (MIP) formulation for \eqref{eq:mainprob}, in the case that the support of $\rveta$ and $\rvxi$ is finite, that is due to \citet{Goel2006}.
A key concept for this formulation is whether two possible scenarios $\eta,\eta' \in \Eta$ can be distinguished from each other for a given a selection of probing decisions $S \subseteq [n]$.


\begin{definition}
Given a set $S \subseteq [n]$, scenarios $\eta \in \Eta$ and $\eta' \in \Eta$ are \emph{indistinguishable with respect to $S$}  if  \[ \sum_{j \in S} |\eta_j - \eta^\prime_j| = 0. \]
\end{definition}
We denote the set of all (unordered) pairs of outcomes of $\rveta$ as 
\begin{equation*}
\Eta^2 := \{ \{\eta,\eta'\} \in \Eta \times \Eta : \eta \neq \eta^\prime \}, 
\end{equation*} and the set of all pairs of random vectors indistinguishable with respect to $S$ as $\mathcal{I}(S) \subseteq \Eta^2$.

The MIP model starts by writing a non-anticipative formulation of \eqref{eq:mainprob}.  If the $S \subseteq [n]$ components of $\rveta$ are probed, the decision $y$ can depend (only) on the observed outcome $\seta_S$.   
Let $y(\eta) : \Eta \mapsto \mathbb{R}^p$ be a mapping corresponding to the first stage decisions $y$ in \eqref{eq:2sp}.
In order for our decisions $y(\eta)$ to be appropriately nonanticipative (and the mapping $y(\cdot)$ to be measurable), we require that 
\begin{equation}
\label{eq:na}
    y(\eta) = y(\eta') \quad \mathbb{P}-a.e. \ \{\eta,\eta'\} \in \mathcal{I}(S).
\end{equation}
That is, if the $S$ components of $\rveta$ are probed, and two outcomes  $\eta,\eta' \in \Eta$ are indistinguishable with respect to $S$, then the same first stage decisions must be made. 
Note that $\mathcal{I}(\emptyset) = \Eta^2$, so in this case~\eqref{eq:na} reduces to ensuring that $y(\eta) = y$ is a constant-valued mapping, and the problem reduces to a standard two-stage stochastic program.  

We introduce binary variables $x_j \in \{0,1\}, j \in [n]$ to indicate whether or not to probe the $j$th component of $\rveta$, and define $S(x) := \{j \in [n] : x_j = 1\}$ as the set of selected probing decisions.
We define $z(\xi) : \Xi \mapsto \mathbb{R}^\zd$ as a measurable mapping corresponding to the second-stage decisions $z$ in \eqref{eq:2sp}.







The problem \eqref{eq:mainprob} can then be formulated as
\begin{subequations}
\label{eq:na-notmip}
\begin{alignat}{2}
    \max_{x,y,z} \quad &\alpha^\tp x + \Expect_{\rveta, \rvxi} \Bigl[\beta^\tp y(\rveta) + \gamma(\rvxi)^\tp z(\rvxi) \Bigr] \\
    \mbox{s.t. } \quad & A y(\eta) = b &&\mathbb{P} \mbox{-} a.e. \ \seta \in \Eta, \\
    & T(\sxi)y(\eta) + W(\sxi)z(\sxi) = h(\sxi) && \mathbb{P} \mbox{-} a.e. \ (\seta, \sxi) \in \Eta \times \Xi,   \\
    &y(\seta) = y(\seta') && \mathbb{P} \mbox{-} a.e. \ (\seta, \seta') \in \mathcal{I}(S(x)),  \label{eq:na-bad} \\
    & x \in \{0, 1\}^{\xd}\\
    & y(\seta) \in \mathbb{R}^{\hat{\yd}}_+ \times \mathbb{Z}^{\yd-\hat{\yd}}_+  &&\mathbb{P} \mbox{-} a.e. \ \seta \in \Eta, \\
    & z(\sxi) \in \mathbb{R}^{\hat{\zd}}_+ \times \mathbb{Z}^{\zd-\hat{\zd}}_+  &&\mathbb{P} \mbox{-} a.e. \ \sxi \in \Xi.
\end{alignat}
\end{subequations}


In~\eqref{eq:na-notmip}, the nonanticapativity constraints~\eqref{eq:na-bad} to be enforced depend on the probing decision vector $x$.  We next discuss how this can be formulated using linear constraints under the assumption that $Y$ is bounded.  
Specifically, for every pair of outcomes $(\eta, \eta') \in \Eta^2$, if the logical implications 
\begin{align*}
\sum_{j \in [n]} |\seta_j - \seta'_j| x_j \leq 0 &\Rightarrow y_i(\seta) - y_i(\seta') \leq 0 \ \forall i \in [p] \\
\sum_{j \in [n]} |\seta_j - \seta'_j| x_j \leq 0 &\Rightarrow -y_i(\seta) + y_i(\seta') \leq 0 \ \forall i \in [p] \\
\end{align*}
are modeled, then when $\eta$ and $\eta'$ are the same on the observed variables, the first stage decision $y$ is forced to take the same values. 
For each pair of outcomes $\{\eta,\eta'\} \in \Eta^2$, define 
\[ \varepsilon(\seta,\seta') := \min_{j \in [n] : |\seta_j - \seta'_j| > 0} \{ |\seta_j - \seta'_j| \} > 0 \] as the smallest amount by which the two realizations differ.   For $i \in [p]$, define 
\[ R_i := \max_{y \in Y} \{y_i\} - \min_{y \in Y} \{y_i\} \]
as the range of the decision variable $y_i$.  We finally define the big-M value $M_i(\seta,\seta') := [\varepsilon(\seta,\seta')]^{-1}R_i.$
With these definitions, we can write a MIP formulation of~\eqref{eq:mainprob} as 
\begin{subequations}
\label{eq:na-mip}
\begin{alignat}{2}
    \max_{x,y,z} \quad &\alpha^\tp x + \Expect_{\rveta, \rvxi} \Bigl[\beta^\tp y(\rveta) + \gamma(\rvxi)^\tp z(\rvxi) \Bigr] \\
    \mbox{s.t. } \quad & A y(\eta) = b &&\mathbb{P} \mbox{-} a.e. \ \seta \in \Eta, \\
    & T(\sxi)y(\eta) + W(\sxi)z(\sxi) = h(\sxi) && \mathbb{P} \mbox{-} a.e. \ (\seta, \sxi) \in \Eta \times \Xi,   \\
    & y_i(\seta) - y_i(\seta') - M_i(\seta,\seta') \sum_{j \in [n]} |\seta_j-\seta'_j| x_j \leq 0 &&\mathbb{P} \mbox{-} a.e. \ \{\seta,\seta'\} \in \Eta^2,  \label{eq:nac-mip-1}\\
    & -y_i(\seta) + y_i(\seta') - M_i(\seta,\seta') \sum_{j \in [n]} |\seta_j-\seta'_j| x_j \leq 0 &\ \ &\mathbb{P} \mbox{-} a.e. \ \{\seta,\seta'\} \in \Eta^2,  \label{eq:nac-mip-2}\\    
    & x \in \{0, 1\}^{\xd}\\
    & y(\seta) \in \mathbb{R}^{\hat{\yd}}_+ \times \mathbb{Z}^{\yd-\hat{\yd}}_+  &&\mathbb{P} \mbox{-} a.e. \ \seta \in \Eta, \\
    & z(\sxi) \in \mathbb{R}^{\hat{\zd}}_+ \times \mathbb{Z}^{\zd-\hat{\zd}}_+  &&\mathbb{P} \mbox{-} a.e. \ \sxi \in \Xi.
\end{alignat}
\end{subequations}

If the support of the random variables $\rveta$ and $\rvxi$ is finite, then~\eqref{eq:na-mip} is a MIP problem with a finite number of constraints and variables.  Note, however, that from~\eqref{eq:nac-mip-1} and~\eqref{eq:nac-mip-2}, there are two constraints for every \emph{pair} of distinct outcomes of $\rveta$, so this formulation is only tractable if $|\Eta|$ is quite small.
As mentioned in the introduction, some authors such as \citet{boland.et.al:16} have done work to identify a minimal set of nonanticipative constraints~\eqref{eq:nac-mip-1} and~\eqref{eq:nac-mip-2} that still results in a valid formulation of~\eqref{eq:mainprob}, but the number of such constraints may still be extremely large.

\section{Branch-and-Bound Algorithms} \label{sec:Branch_and_Bound}

In this section, we propose two branch-and-bound algorithms to solve~\eqref{eq:mainprob}.  We describe the algorithms under the assumption that the quantity $F(S)$ can be computed exactly for any $S \subseteq [n]$, for example, when the joint support of $(\rveta,\rvxi)$ is finite and not too large. 
We discuss in Section~\ref{sec:Sampling} how sampling can be used to yield statistical approximations when $F(S)$ cannot be computed exactly.

Upper bounds in our branch-and-bound method are based on an information-relaxation concept generalizing the notion of perfect information. Specifically, the relaxations follow from the following observation.

\begin{lemma}
\label{lem:st}
If $S \subseteq T \subseteq [\xd]$, then $F(S) \leq F(T)$.
\end{lemma}
\begin{proof}
Using the definition of $F(T)$ we have
\begin{align*}
F(T) &= \Expect_{\rveta_T} \Bigg[ \max_{y \in Y} \Expect_{\rvxi} \Bigl[ \beta^\tp y + Q(y,\rvxi)|\rveta_T  \Bigr] \Bigg] \\
    &= \Expect_{\rveta_S} \Expect_{\rveta_{T \setminus S}} \Bigg[ \max_{y \in Y} \Expect_{\rvxi} \Bigl[ \beta^\tp y + Q(y,\rvxi)|\rveta_S, \rveta_{T \setminus S}  \Bigr] \Bigg] \\
    &\geq \Expect_{\rveta_S}  \max_{y \in Y} \Expect_{\rveta_{T \setminus S}} \Bigg[ \Expect_{\rvxi} \Bigl[ \beta^\tp y + Q(y,\rvxi)|\rveta_S, \rveta_{T \setminus S}  \Bigr] \Bigg] \\
       &= \Expect_{\rveta_S} \Bigg[ \max_{y \in Y}  \Expect_{\rvxi} \Bigl[ \beta^\tp y + Q(y,\rvxi)|\rveta_S  \Bigr] \Bigg] = F(S) .
\end{align*}
\qed
\end{proof}

The set $[\xd]$ includes the index of all entries in $\eta$, so $F([\xd]) \geq \max \{  F(S) : S \subseteq [\xd] \}$. Recalling that $\alpha: 2^{[n]} \rightarrow \mathbb{R}_+$, is a non-negative set function, if we simply drop the term $-\alpha(S)$ in \eqref{eq:mainprob} we obtain an upper bound on the optimal solution value that is known as the \emph{perfect information bound} $F([\xd])$:
\begin{equation*} 
    F([\xd]) \geq \zpesp.
\end{equation*}
On the other hand, probing all components of $\rveta$ is one feasible solution to \eqref{eq:mainprob}, which yields the following lower bound on the optimal value:
\begin{equation*} 
    F([\xd])- \alpha([\xd])  \leq \zpesp.
\end{equation*}

\subsection{Single-Element Branching}
\label{sec:single-element-branching}

In our first branch-and-bound algorithm, after evaluating the upper and lower bounds at a node, a single component of $j \in [n]$ of $\rveta$ is selected, and two new nodes are created, one where $\rveta_j$ is probed, and one where $\rveta_j$ is \emph{not} probed.  Thus, 
a node $P$ in the branch-and-bound tree is defined by a subset of $S^P_1 \subseteq [n]$ of components that are probed and a disjoint subset $S^P_0$ of components that are \emph{not} probed. The following lemma provides lower and upper bounds on the optimal value of the node problem given these restrictions.

\begin{lemma}
\label{lem:svbr}
Let $S_0^P, S_1^P \subseteq [n]$ with $S_0^P \cap S_1^P =\emptyset$.   Then, 
\begin{align*}
F([\xd] \setminus S^P_0)-\alpha([\xd] \setminus S^P_0)  & \leq
\max \{  F(S) -\alpha(S)\, : \, S \supseteq S^P_1, S \cap S^P_0 = \emptyset \} \\ 
& \leq F([\xd] \setminus S^P_0)-\alpha(S^P_1).
\end{align*}
\end{lemma}
\begin{proof}
The lower bound is obtained by observing that the set $S = [\xd] \setminus S^P_0$ satisfies $S \supseteq S^P_1$ and $S \cap S^P_0 = \emptyset$. The upper bound follows from the monotonicity of $\alpha$, $\alpha(S^P_1) \leq \alpha(S)$ for any  $S \supseteq S^P_1$, and by applying Lemma~ \ref{lem:st} to conclude that $F([\xd] \setminus S^P_0) \geq F(S)$ for any $S$ with $S \cap S^P_0 = \emptyset$. \qed
\end{proof}



\def\textA{{\scriptsize$S_0 = \emptyset, \ S_1 = \emptyset$}
\nodepart{two}
{\scriptsize$\text{ub} = \alpha(\emptyset) + F(\{1,2,3\})$}\\
{\scriptsize$\text{lb} = \alpha(\{1,2,3\}) + F(\{1,2,3\})$}\\
}

\def\textB{{\scriptsize$S_0 = \{ 2\}, \ S_1 = \emptyset$}
\nodepart{two}
{\scriptsize$\text{ub} = \alpha(\emptyset) + F(\{1,3\})$}\\
{\scriptsize$\text{lb} = \alpha(\{1,3\}) + F(\{1,3\})$}\\
}

\def\textC{{\scriptsize$S_0 = \emptyset, \ S_1 = \{ 2\}$}
\nodepart{two}
{\scriptsize$\text{ub} = \alpha(\{2\}) + F(\{1,2,3\})$}\\
{\scriptsize$\text{lb} = \alpha(\{1,2,3\}) + F(\{1,2,3\})$}\\
}

\def\textD{{\scriptsize$S_0 = \{ 1,2\}, \ S_1 = \emptyset$}
\nodepart{two}
{\scriptsize$\text{ub} = \alpha(\emptyset) + F(\{3\})$}\\
{\scriptsize$\text{lb} = \alpha(\{3\}) + F(\{3\})$}\\
}

\def\textE{{\scriptsize$S_0 = \{ 2\}, \ S_1 = \{ 1\}$}
\nodepart{two}
{\scriptsize$\text{ub} = \alpha(\{1\}) + F(\{1,3\})$}\\
{\scriptsize$\text{lb} = \alpha(\{1,3\}) + F(\{1,3\})$}\\
}

\tikzset{>=stealth,parent node/.style={rectangle split, rectangle split parts=2,align=center,text width=4.2cm,draw,node distance=.6cm and -1.9cm}}

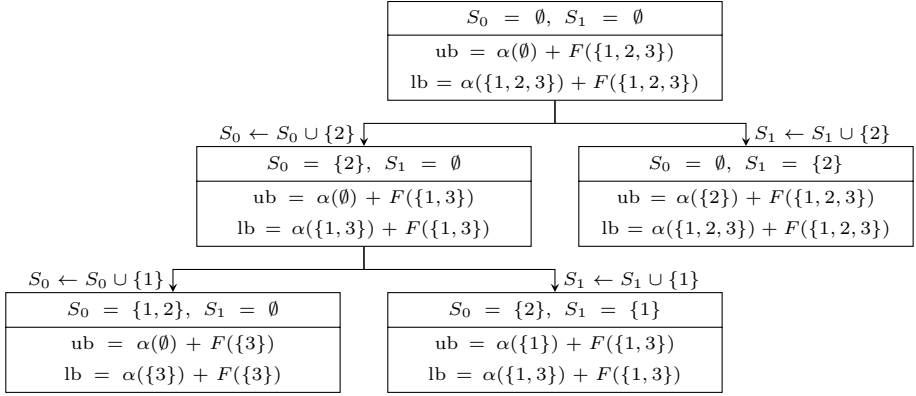
\begin{figure}
\centering
\begin{tikzpicture}
\node[parent node](N0){\textA};
\node[parent node,below left =of N0](N1){\textB};
\node[parent node,below right =of N0](N2){\textC};

\node[parent node,below left =of N1](N3){\textD};
\node[parent node,below right=of N1](N4){\textE};

\draw[->](N0.south)--+(0,-0.3)-|(N1)node[left,near end]{\scriptsize $S_0 \leftarrow S_0 \cup \{2\}$};
\draw[->](N0.south)--+(0,-0.3)-|(N2)node[right,near end]{\scriptsize $S_1 \leftarrow S_1 \cup \{2\}$};

\draw[->](N1.south)--+(0,-0.3)-|(N3)node[left,near end]{\scriptsize  $S_0 \leftarrow S_0 \cup \{1\}$};
\draw[->](N1.south)--+(0,-0.3)-|(N4)node[right,near end]{\scriptsize $S_1 \leftarrow S_1 \cup \{1\}$};

\end{tikzpicture}
\caption{Partial Bound Calculations for $n=3$}
\label{fig:b&b_example}
\end{figure}



      

After evaluating node $P$ of the branch-and-bound tree, if its upper bound is not smaller than the value of a known feasible solution, then we must \emph{branch}, by selecting an unfixed component $j \in [n] \setminus S_1^P \setminus S_0^P$, and creating two child nodes $P^+$ and $P^-$, with 
\begin{alignat*}{2}
S_1^{P^+} &:= S_1^P \cup \{j\} \quad & S_0^{P^+} & := S_0^P, \\
S_1^{P^-} &:= S_1^P \quad & S_0^{P^-} & := S_0^P \cup \{j\}.\\
\end{alignat*}

The single-element branch-and-bound algorithm is initialized with a root node $A$ defined by $S_0^A = S_1^A = \emptyset$ and then proceeds to create more nodes via the single-element branching procedure. Figure \ref{fig:b&b_example} presents an example of a partial branch-and-bound tree with the bound calculations for an instance with $n=3$. 
Note that if $S_0^P \cup S_1^P = [n]$ then the lower and upper bounds in Lemma \ref{lem:svbr} coincide. Thus, there will be at most $2^n$ leaf nodes in the tree, so the branch-and-bound algorithm is finite.  An important aspect of the proposed bounding scheme is that evaluating the bounds for the child node where the probing candidate \emph{must} be probed (it is added to the set $S^{P^+}_1$) requires only an evaluation of $\alpha()$, since all other components of the bounds were computed at the parent node. 

\subsubsection*{Branching variable selection}


The choice of branching variable index among the unfixed components $C := [n] \setminus S_1^P \setminus S_0^P$ can significantly impact the size of the search tree, and we would like to select a candidate for which the upper bounds of the two child nodes decrease significantly. For each $j \in C$, we seek to estimate the bound that will be obtained if we create child nodes $P^+$ and $P^-$ using $j$ as the branching variable.
For the node $P^+$, the computed upper bound changes by 
\[ \Delta_j^+ :=  \alpha(S^P_1) - \alpha(S^P_1 \cup \{j\}) \leq 0. \]
By evaluating $\alpha(S_1^P \cup \{j\})$ for each unfixed variable $j \in C$, we can decide which unfixed variable will have the most impact for the node $P^+$. 
For the node $P^-$, the computed upper bound changes by 
\[ \Delta_j^- := F([\xd] \setminus S^P_0) -  F([\xd] \setminus (S^P_0 \cup \{j\})).  \]
Thus, we could like to estimate the impact that \emph{not} knowing the outcome of the random variable $\bfeta_j$ will have on the upper bound, relative to the current upper bound.
Letting $S = [n] \setminus S_0^P$, the first term in the definition of $\Delta_j^-$ is $F(S) = \Expect_{\rveta_S} [R(\rveta_{S})]$ and has been calculated (or estimated by sampling) at the current node, but computing $F([\xd] \setminus (S^P_0 \cup \{j\}))$ for each $j \in C$ would be computationally demanding. Thus, we instead seek to use information obtained when evaluating $F(S)$ to estimate the relative importance of knowing each random component $\rveta_j$ in terms of its impact on the upper bound. Thus, rather than use $\Delta_j^-$ exactly, we use an estimate $\hat{\Delta}_j^-$ and consider two options for this estimate.

The first option is applicable only in the case where each $\rveta_j$ is a Bernoulli random variable. In this option, we define $\hat{\Delta}_j^-$ to be an approximation of the value
\[  \Expect_{\rveta_S} [R(\rveta_{S})|\rveta_j = 1] - \Expect_{\rveta_S} [R(\rveta_{S})|\rveta_j = 0]. \]
In the second option, which applies generally, we define $\hat{\Delta}_j^-$ to be an approximation of the covariance between $\rveta_j$ and $R(\rveta_S)$. In the case that $F(S)$ is estimated via sampling as described in Section \ref{sec:Sampling}, these estimates can be computed using the same information used to evaluate $F(S)$. The details are described in the Appendix. 

To combine the $\Delta^+_j$ and $\hat{\Delta}^-_j$ values into a score for each candidate $j \in C$ we first normalize them as follows
\[ 
\zeta^{+}_{j} = \frac{\Delta^+_j - \min_{j' \in C} \Delta^+_{j'}}{\max_{j' \in C} \Delta^+_{j'} - \min_{j' \in C} \Delta^+_{j'}}, \quad 
\zeta^{-}_{j} = \frac{\hat{\Delta}^-_j - \min_{j' \in C} \hat{\Delta}^-_{j'}}{\max_{j' \in C} \hat{\Delta}^-_{j'} - \min_{j' \in C} \hat{\Delta}^-_{j'}},
\]
then define 
\begin{equation}
    \label{eq:def_psi}
    \psi_j := \zeta_j^+ + \zeta_j^-,
\end{equation} 
and finally choose the branching variable $j^*$ according to $ j^{\ast} = \mathbf{argmax}_{j \in C} \psi_j.$

\subsection{Multi-element branching}
\label{sec:multi-element-branching}

In our second branch-and-bound algorithm, we characterize a node $P$ by a subset $S^P_0 \subseteq [n]$ of components of $\rveta$ that will \emph{not} be probed, and a collection of disjoint subsets $\mathcal{S}^P = \{S^P_k : k \in [n_P]\}$, where for each $k \in [n_P]$ \emph{at least one} of the components of $\rveta_{S^P_k}$ will be probed. This method requires the additional assumption that $\alpha(\cdot)$ is modular:  $\alpha(S) = \sum_{i \in S} \alpha(\{i\})$ for all $S \subseteq [n]$.

The following lemma provides lower and upper bounds on the optimal value of the node problem given these restrictions.
\begin{lemma}
\label{lem:subsetbr}
Assume that $\alpha(\cdot)$ is modular. Let $S_0^P \subseteq [n]$, and $\mathcal{S}^P = \{S^P_k : k \in [n_P]\}$ be such that $S_i^P \cap S_j^P = \emptyset \ \forall i\neq j \in [n_P] \times [n_P]$ and $S_0^P \cap S_i^P = \emptyset \ \forall i \in [n_P]$. Then, 
\begin{align*}
F([\xd] \setminus S^P_0) & -\alpha([\xd] \setminus S^P_0)\\
& \leq
\max \{ F(S) -\alpha(S) \, : \, S \cap S^P_0 = \emptyset, S \cap S^P_k \neq \emptyset, k \in [n_P] \} \\ 
& \leq F([\xd] \setminus S^P_0) - \sum_{k \in [n_P]} \min \{ \alpha_j : j \in S^P_k \} .
\end{align*}
\end{lemma}
\begin{proof}
The lower bound is obtained by observing that the set $S = [\xd] \setminus S^P_0$ satisfies $S \cap S^P_k \neq \emptyset$ for $k \in [n_P]$ and $S \cap S^P_0 = \emptyset$.  By the monotonicity and modularity of $\alpha$,  
\[ \sum_{k\in [n_P]} \min\{ \alpha_j : j \in S^P_k \} \leq \alpha(S) \]
for any  $S $ with $S \cap S^P_k \neq \emptyset$ for all $k \in [n_P]$.  Using Lemma \ref{lem:st} to conclude that $F([\xd] \setminus S^P_0) \geq F(S)$ for any $S$ with $S \cap S^P_0 = \emptyset$ yields the desired upper bound. \qed
\end{proof}

To begin the multi-element branching tree, the root node $A$ is created and evaluated with $S_0^A = \emptyset, \mathcal{S}^A = \emptyset$, which yields the perfect information bound.  The first branch creates two nodes $L$ and $R$ enforcing the dichotomy that either no elements of $\rveta$ are probed ($S_0^L = [n], \mathcal{S}^L = \emptyset$), or at least one component of $\rveta$ is probed, ($S_0^R = \emptyset, \mathcal{S}^R = \{ [n] \}$).
At all subsequent nodes of the branch-and-bound tree, branching on an active node $P$ is done by selecting a subset element $T \in \mathcal{S}^P$ and a subset $K \subset T$.  Since at least one element of $T$ must be selected for probing, we know the following trichotomy: 
\begin{enumerate}
    \item If no elements of $K$ are probed, then at least one element of $T \setminus K$ must be probed;
    \item If no elements of $T \setminus K$ are probed, then at least one element of $K$ must be probed;
    \item Otherwise at least one element from both $T$ and $T \setminus K$ must be probed.
\end{enumerate}
This translates into three new child nodes $P_1$, $P_2$, $P_3$ of $P$ defined by 
\begin{enumerate}
    \item $S_0^{P_1} = S_0^P \cup K.  \quad \mathcal{S}^{P_1} = \mathcal{S}^P \setminus \{T\} \cup \{T \setminus K\}$.
    \item $S_0^{P_2} = S_0^P \cup (T \setminus K).  \quad \mathcal{S}^{P_2} = \mathcal{S}^P \setminus \{T\} \cup \{ K\}$.
    \item $S_0^{P_3} = S_0^P. \quad \mathcal{S}^{P_3} = \mathcal{S}^P \setminus \{T\} \cup \{ K\} \cup \{T \setminus K\}$.
\end{enumerate}

Figure \ref{fig:b&b_multibranch_example} presents an example of a partial branch-and-bound tree obtained using multi-element branching for an example with $n=5$. Compared to single-element branching, this approach tends to lead to subproblems with more elements in the set $S_0^P$ higher in the tree, which may improve the bounds more quickly as the quantity $F([n] \setminus S_0^P)$ is a key component of the node bounds.


\def\textMultiA{{\scriptsize$S_0 = \emptyset,$\\$\mathcal{S} = \emptyset$}
}

\def\textMultiB{{\scriptsize$S_0 = \{1,2,3,4,5\},$\\$\ \mathcal{S} = \emptyset$}
}

\def\textMultiD{{\scriptsize$S_0 = \emptyset,$\\$\ \mathcal{S} = \{\{ 1,2,3,4,5\}\}$}
}

\def\textMultiE{{\scriptsize$S_0 = \{1,3,5\},$\\ $\mathcal{S} = \{\{2,4\}\}$}
}

\def\textMultiF{{\scriptsize$S_0 = \{2,4\},$\\ $\mathcal{S} = \{\{1,3,5\}\}$}
}

\def\textMultiG{{\scriptsize$S_0 = \emptyset,$\\$\mathcal{S} = \{\{1,3,5\}, \{2,4\}\}$}
}

\def\textMultiH{{\scriptsize$S_0 = \{1,2,4,5\},$\\ $\mathcal{S} = \{\{3\}\}$}}
\def\textMultiI{{\scriptsize$S_0 = \{2,3,4\},$\\ $\mathcal{S} = \{\{1,5\}\}$}}
\def\textMultiJ{{\scriptsize$S_0 = \{2,4\},$\\ $\mathcal{S} = \{\{1,5\}, \{3\}\}$}}

\tikzset{>=stealth,parent node/.style={align=center,text width=0.22\textwidth,draw,node distance=.8cm and 0cm}}

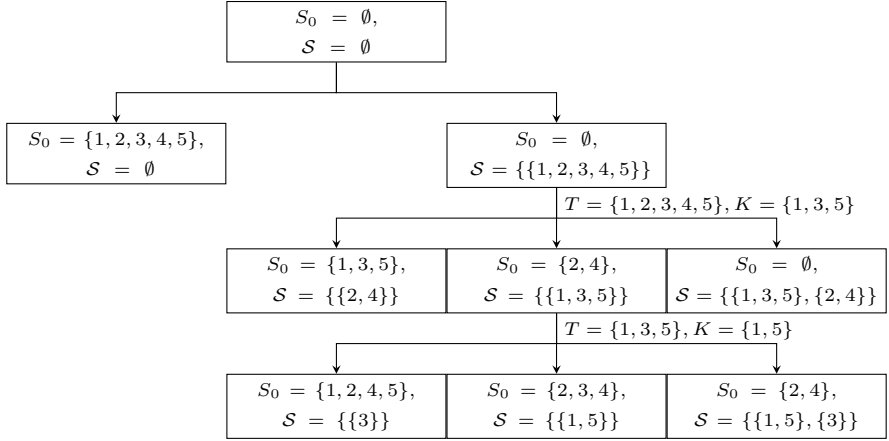
\begin{figure}
\centering
\begin{tikzpicture}
\node[parent node](N0){\textMultiA};
\node[parent node,below left =of N0](N1){\textMultiB};
\node[parent node,below right =of N0](N3){\textMultiD};

\node[parent node,below left =of N3](N4){\textMultiE};
\node[parent node,below =of N3](N5){\textMultiF};
\node[parent node,below right =of N3](N6){\textMultiG};

\node[parent node,below left =of N5](N7){\textMultiH};
\node[parent node,below =of N5](N8){\textMultiI};
\node[parent node,below right =of N5](N9){\textMultiJ};

\draw[->](N0.south)--+(0,-0.4)-|(N1)node[left,near end]{};
\draw[->](N0.south)--+(0,-0.4)-|(N3)node[right,near end]{};

\draw[->](N3.south)--+(0,-0.4) node[right, yshift=5pt] {\scriptsize $T = \{1,2,3,4,5\}, K=\{1,3,5\}$} -| (N4)node{};
\draw[->](N3.south)--+(0,-0.4)-|(N5)node[right,near end]{};
\draw[->](N3.south)--+(0,-0.4)-|(N6)node[right,near end]{};

\draw[->](N5.south)--+(0,-0.4) node[right, yshift=5pt] {\scriptsize $T = \{1,3,5\}, K=\{1,5\}$} -| (N7)node{};
\draw[->](N5.south)--+(0,-0.4)-|(N8)node[right,near end]{};
\draw[->](N5.south)--+(0,-0.4)-|(N9)node[right,near end]{};

\end{tikzpicture}
\caption{Example partial branch-and-bound tree with $n = 5$ multi-element branching.}
\label{fig:b&b_multibranch_example}
\end{figure}

\subsubsection*{Branching decisions}

Branching in the multi-element branching method requires two decisions.  A branching set $T \in {\mathcal{S}^P}$ must be selected, and the elements of $T$ must be partitioned.  In our implementation, we select the branching set $T$ from ${\mathcal{S}^P}$ as the set with the largest cardinality.
To partition the elements of the branching set $T$, we use the importance estimates $\psi_j$ introduced in equation~\eqref{eq:def_psi}.  The set is partitioned 
by choosing $K$ to heuristically minimize the difference $|\sum_{j \in K} \psi_j - \sum_{j \in T \setminus K} \psi_j|$, which tries to ensure that both child nodes are of ``equal'' importance. Specifically, we sort the elements of $T$ in decreasing order of their scores $\psi_j$. Let $j_1, j_2, \dots, j_{|T|}$ be the elements of $T$ sorted such that $\psi_{j_1} \geq \psi_{j_2} \geq \dots \geq \psi_{j_{|T|}}$. Then we partition $T$ into $K$ and $T \setminus K$  by setting $K= \{j_k : k \text{ is odd}\}$.

\section{Sampling}
\label{sec:Sampling}

In many applications, the joint support of the random variables $(\rveta, \rvxi)$ will be too large allow for the exact evaluation of $F(S)$. We present two methods for estimating an upper bound on $\zpesp$ using sampling, which vary according to whether the branch-and-bound algorithm is run on an approximation built using a single sample (external sampling -- Section \ref{sec:extsamp}) or whether sampling is done repeatedly throughout the branch-and-bound algorithm (internal sampling -- Section \ref{sec:intsamp}). We also describe in Section \ref{sss:slb} how sampling can be used to estimate a lower bound on $\zpesp$.

\newcommand{\xrub}{\mathbf{R}_N}
\newcommand{\xfub}{\mathbf{F}_N}

\subsection{External Sampling---Sample Average Approximation}
\label{sec:extsamp}

A classical way to obtain bounds on the optimal solution value of a stochastic program is via sample-average approximation (SAA).  In SAA, $(\seta^k,\sxi^k)$, for $k \in [N]$ are jointly randomly sampled according to the distribution $\mathbb{P}$, and an approximation to the problem assuming each of these scenarios is equally likely is created. 


Consider a fixed sample and let   $S \subseteq [n]$. Let $\hat{\Eta}_S = \bigcup_{k=1}^N \seta^k_S$ be the set of unique subvectors of $\seta^k$ in the sample. For each $\seta_S \in \hat{\Eta}_S$, define $\Omega_N(\seta_S)=\{ k \in [N] : \seta^k_S = \seta_S \}$. 
Then, on this random sample, the sample estimate $\xfub(S)$ of $F(S)$ is computed by solving the following two-stage stochastic program for each $\seta_S \in \hat{\Eta}_S$: 
\begin{equation}
\label{eq:rncalc_saa}
\xrub(\seta_S) = \max_{y \in Y} \beta^\tp y + |\Omega_N(\seta_S)|^{-1} \sum_{k \in \Omega_N(\seta_S)} Q(y,\sxi^k)
\end{equation}
where $Q(y,\sxi^k)$ is defined in \eqref{eq:recourse-function}.
The value $\xfub(S)$ is then computed as
\[ \xfub(S) = 
{N}^{-1} \sum_{\seta_S \in \hat{\Eta}_S} |\Omega_N(\seta_S)| \xrub(\seta_S). \]
Finally, the SAA optimal value for a given sample is defined as
\begin{equation}
\label{eq:saa}
 \bfv_N = \max\{  \xfub(S) - \alpha(S) : S \subseteq [n] \}.
\end{equation}

Since $\xfub(S)$ can be computed for any $S \subseteq [n]$ by solving a collection of two-stage stochastic programs with finite support, the sampled approximation problem \eqref{eq:saa} can be solved using the branch-and-bound method of Section~\ref{sec:Branch_and_Bound}.  Any method can be used to solve each of the two-stage stochastic programs \eqref{eq:rncalc_saa}. In our implementation, we use a commercial MIP solver to solve the extensive form \cite{SPbook_birge1997}, but decomposition methods such as the L-shaped algorithm \cite{lshaped:siam69,laporte1993integer} or dual decomposition \cite{caroe1999dual} may also be applied when applicable.


The optimal solution value of the sampled instance $\bfv_N$ is an outcome of a random variable, as it depends on the randomly drawn sample of size $N$.  It is well-known that the expected value of the random variable $\bfv_N$ provides a biased estimate of the true objective value: 
$\mathbb{E}[ \bfv_N ] \geq \zpesp$.
In order to estimate $\Expect[\bfv_N]$, and hence estimate an upper bound of $\zpesp$, we can replicate the computations, as suggested by \citet{mak.morton.wood:99}.
Specifically, we draw $L$ independent batches of joint samples $(\seta^{k,\ell},\sxi^{k,\ell})$, for $k \in [N], \ell \in [L]$ from the distribution $\mathbb{P}$.  
Note that since the sampled elements \emph{within} a batch are not required to be independent, variance reduction techniques such as Latin hypercube sampling (LHS) \citep{McKay1979LHS,Freimer2012LHS} could be employed to generate the sample for each fixed batch $\ell \in [L]$.
We define $\bfv^\ell_N$ to be the optimal solution value of~\eqref{eq:saa} coming from the $\ell$th batch of samples and define the average of the $\bfv^\ell_N$ values as
\begin{equation}
\label{eq:saa_avg_val}
\mathbf{\Upsilon}_{N,L} := L^{-1} \sum_{\ell \in [L]} \bfv^\ell_N 
\end{equation}
The estimator $\mathbf{\Upsilon}_{N,L}$ is an unbiased estimate of $\Expect[ \bfv_N]$
and by the Central Limit Theorem, we know that 
\[ 
    \sqrt{L}(\mathbf{\Upsilon}_{N,L} - \Expect[\bfv_N]) \Rightarrow \mathsf{N}(0,\operatorname{Var}(\bfv_N)) \mbox{ as } L \rightarrow \infty.
\]
The sample estimator of the variance $\operatorname{Var}(\bfv_N)$ is
\begin{equation}
\label{eq:saa_var_val}
\mathbf{s}^2_{N,L} := (L-1)^{-1} \sum_{\ell \in [L]} (\bfv_N^\ell - \Upsilon_{N,L})^2, 
\end{equation}
which can be used to compute an approximate $(1-\alpha)$ confidence upper bound of $\Expect[\bfv_N]$ as 
\begin{equation}
\label{eq:saa_ci_ub}
    \mathbf{\Upsilon}_{N,L} + L^{-1/2} t_{L-1,\alpha} \mathbf{s}_{N,L},
\end{equation}
where $t_{L-1,\alpha}$ is the $\alpha$ critical value of the Student's t distribution with $L-1$ degrees of freedom.
It is also shown in \citet{mak.morton.wood:99} that 
$\mathbb{E} \bfv_N  \geq \mathbb{E} \bfv_{N+1} \geq \zpesp$, so increasing the sample size $N$ reduces the bias of the estimate of $\zpesp$.  

\subsubsection*{Re-using information.}
When running the branch-and-bound algorithm with a fixed sample $(\seta^k,\sxi^k)_{k=1}^N$ it frequently occurs that the subset $\Omega_N(\eta_S)$ used in the calculation of $\xrub(\seta_S)$ is identical to the subset $\Omega_N(\eta'_{S'})$ for some previously observed $\eta'_{S'} \neq \eta_{S}$. For example, at nodes in which the upper bound is computed based on a set $S$ that consists of many elements of $[n]$ (as often happens in early nodes in the branch-and-bound tree) the sets $\Omega_N(\eta_S)$ have small cardinality, or may even be singletons, and hence repeat often. 
Thus, we suggest storing the computed values of $\xrub(\seta_S)$ in a hash table with the key defined by a binary encoding of the set $\Omega_N(\eta_S)$. Thus, any time we need to evaluate $\xrub(\seta_S)$ we first determine the set $\Omega_N(\eta_S)$ and check this hash table to see if the the required two-stage stochastic program has already been solved.  In Section~\ref{sec:hashing_good}, we report on the significant computational savings that occur from this simple observation.   

\subsection{Internal Sampling}
\label{sec:intsamp}

A limitation of the external sampling approach is that within a single replication, the same sample is used to estimate $F(S)$ for all $S$ that are encountered within the branch-and-bound algorithm, which prevents taking advantage of the specific subset of observed elements of $\rveta_S$ when estimating $F(S)$. In particular, if the support of $\rveta_S$ is infinite then the set $\hat{\Eta}_S$ of unique vectors $\eta^k_S$ in any randomly-drawn finite sample will include the full sample so that the set $\Omega_N(\seta_S)$ is a singleton for each $\seta_S \in \hat{\Eta}_S$. This implies that the estimate $F_N(S)$ will revert to the perfect information bound for {\it any finite sample} $N$, making it impossible for the external sampling approach to improve upon the perfect information bound. Even if the support of $\rveta_S$ is finite but very large, the same reasoning suggests that the bias of the lower bound $\bfv_N$ used in the external sampling approach may be very large.

To overcome this drawback, sampling can be done separately to estimate an upper bound on each node within the branch-and-bound search. This approach is referred to as {\it internal} sampling because the sampling is done internal to the branch-and-bound search, and can be considered an adaptation of the stochastic branch-and-bound algorithm of \cite{norkin1998branch} to our probing enhanced stochastic programming model. The advantage of internal sampling in this setting is that we can then do {\it conditional} sampling in a nested fashion. To estimate bounds at a node $P$ within the branch-and-bound algorithm we require an estimate of an upper bound of $F([n] \setminus S_0^P)$. A lower bound on the optimal value is obtained by estimating a lower bound of $F(S)$ for a candidate solution $S$. We discuss these cases separately, and also discuss how to obtain an estimate of a global upper bound from a branch-and-bound tree.

\subsubsection{Statistical Upper Bounds on $F(S)$}
\label{sss:sub}



Consider a node in a branch-and-bound tree with associated non-probed set $S_0^P$ and define $S = [n]\setminus S_0^P$.  The main computational task is to estimate a statistical upper bound on $F(S)$.

Let $\eta_S^{k}, k \in [N_1]$ be a sample of $\rveta_S$, and for each observation $\eta_S^k$,  let $\xi^{ki}, i \in [N_2]$ be a {\it conditional} sample of $\rvxi$ taken from the conditional distribution of $\rvxi$ given $\rveta_S = \eta_S^k$.  Figure \ref{fig:intsamp} illustrates this nested sampling procedure.
\begin{figure}
\centering
\caption{Internal Sampling}\label{fig:intsamp}
\begin{forest}
for tree={grow'=0, parent anchor=east, child anchor=west, edge={draw=black}, l sep=20mm}
[
  [$\eta_S^1$
    [{$\xi_1 | \rveta_S = \eta_S^1$}]
    [\vdots, no edge]
    [{$\xi_{N_2} | \rveta_S = \eta_S^1$}]
  ]
  [\vdots, no edge, before drawing tree={draw tree dots between={1}{2}}]
  [\vdots, no edge, before drawing tree={draw tree dots between={1}{2}}]
  [$\eta_S^{N_1}$
    [{$\xi_1 | \rveta_S = \eta_S^{N_1}$}]
    [\vdots, no edge]
    [{$\xi_{N_2} | \rveta_S = \eta_S^{N_1}$}]
  ]
]
\end{forest}
\end{figure}
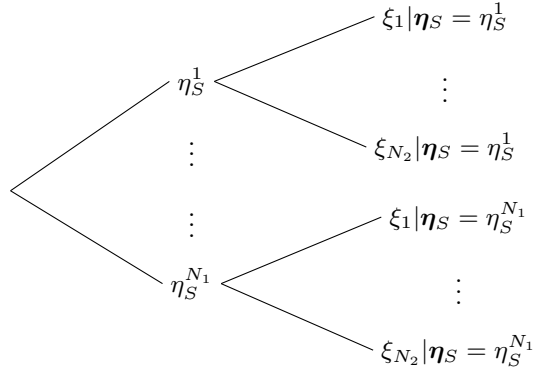

\newcommand{\fub}{\mathbf{\bar{F}}_{N_1,N_2}}
\newcommand{\rub}{\mathbf{\bar{R}}_{N_2}}

An estimate of $F(S)$ is then obtained as
\begin{equation}
\label{eq:fubdef}
 \fub(S) = N_1^{-1} \sum_{k \in [N_1]} \rub(\eta_S^k) 
\end{equation}
where
\[ \rub(\eta_S^k) = \max_{y \in Y} \beta^{\top} y + N_2^{-1} \sum_{i \in [N_2]} Q(y, \xi^{ki}) . \]
Note that computing this estimate $\fub(S)$ requires solving $N_1$ two-stage stochastic programs, each with a sample of size $N_2$

The following result demonstrates that this point estimate of $F(S)$ can provide the basis of a statistical upper bound estimate.
\begin{lemma}
Assume that $\eta^1,\ldots,\eta^{N_1}$ are identically distributed and for each $k \in [N_2]$ the sample $\xi^{k1},\ldots,\xi^{kN_2}$ are identically distributed conditional on $\rveta_S = \eta_S^k$. Then,
\[ \Expect_{\eta^{N_1},\xi^{N_2}}[ \fub(S) ] \geq F(S), \]
where $\Expect_{\eta^{N_1},\xi^{N_2}}$ indicates the expectation is taken with respect to the samples.
\end{lemma}
\begin{proof}
Based on the bias result in sample average approximation \citep{mak.morton.wood:99}, for each $k \in [N_1]$ it holds that 
\[ \Expect_{\xi^{N_2}} \bigl[ \rub(\eta^k_S) \bigr] \geq R(\eta^k_S) \]
where the expectation is taken with respect to the sample $\xi^{k1},\ldots,\xi^{kN_2}$ that is conditional on $\rveta_S = \eta^k_S$.
Thus,
\begin{align*}
\Expect_{\eta^{N_1},\xi^{N_2}} \bigl[ \fub(S) \bigr] &= \Expect_{\eta^{N_1}} \Bigl[ N_1^{-1} \sum_{k \in [N_1]} \Expect_{\xi^{N_2}}[\rub(\eta_S^k)] \Bigr] \\
&\geq \Expect_{\eta^{N_1}} \Bigl[N_1^{-1} \sum_{k \in [N_1]} R(\eta^k_S) \Bigr] \\
& = \Expect_{\rveta_S} [R(\rveta_{S})] = F(S) .
\end{align*}
\qed
\end{proof}

If the sample $\eta^1,\ldots\eta^{N_1}$ is {\it independent} and identically distributed, then the values $\rub(\eta_s^k)$ for $k \in [N_1]$ are iid and hence can be used to compute an approximate statistical upper bound for $F(S)$.
By the Central Limit Theorem, we have
\begin{equation*}
\sqrt{N_1}\Bigl(\fub(S) - \mathbb{E}\bigl[ \rub(\rveta_S)\bigr] \Bigr) \rightarrow N\bigl(0,\operatorname{Var}(\rub(\rveta_S)\bigr),
\end{equation*}
as $N_1 \rightarrow \infty$.
The sample variance of $\rub(\rveta_S)$ is computed as
\begin{equation*}
\mathbf{s}^2_{R_{N_2}}(N_1) = (N_1-1)^{-1} \sum_{k \in [N_1]} \bigl(\rub(\eta_S^k) - \fub\bigr)^2.
\end{equation*}
Thus, for $N_1$ sufficiently large, we can approximate the probability that $F(S)$ exceeds a constant $u$ with the probability that a normal random variable with mean $\fub(S)$ and standard deviation $\mathbf{s}^2_{R_{N_2}}(N_1)$ exceeds $u$. 

Variance reduction techniques such as LHS can be used to improve these estimates similarly as was described for external sampling in Section \ref{sec:extsamp}. This can be accomplished by creating the sample $\eta^1,\ldots\eta^{N_1}$ as a set of {\it batches} where the samples in each batch are generated via LHS, but the batches are generated independently. Then, if the number of batches is sufficiently large, the sample variance of the batch averages is used in place of $\mathbf{s}^2_{R_{N_2}}(N_1)$ for the upper bound estimate.



For purposes of incorporating these upper bounds with the branch-and-bound algorithm of Section~\ref{sec:Branch_and_Bound}, it is crucial to note that these bounds are \emph{statistical} in nature.  Specifically, if the bounds are used to fathom nodes, there is a chance the optimal solution to the problem would be excluded in the search.  Thus, in our branch-and-bound implementation of internal sampling, we \emph{never} fathom a node by bound and simply enumerate nodes until a specified time limit is reached.  In our implementation, we always select to evaluate the node with the largest (statistical) upper bound, emulating the well-known \emph{best-bound} node selection rule from MIP \citep{linderoth.savelsbergh:99}.

\subsubsection*{Exploiting small support of $\rveta$ or $\rvxi$}

If the random variable $\rvxi$ has finite support (e.g., if each element is a Bernoulli random variable) there may be situations where the conditional support of $\rvxi$ given a sample of $\eta_S^k$ of $\rveta_S$ is small enough that $R(\eta_S^k)$ can be computed exactly as described in \eqref{eq:r_etas_s_eval}. In this case, one can replace the sample estimate $\rub(\eta^k_S)$ with its true evaluation $R(\eta_S^k)$ in \eqref{eq:fubdef} to obtain a lower variance estimator, and the remainder of the upper bound derivation is unchanged.

On the other hand, when the support of $\rveta_S$ is small (e.g., in case elements of $\rveta$ are independent Bernoulli trials and $|S|$ is small), we can use this structure to obtain a potentially lower variance estimate of $F(S)$ by summing over these observations. In this situation, let $\eta_S^1, \dots, \eta_S^{N_S}$ be the set of all possible observations of $\rveta_S$, and let $p(\eta^k_S)$ be the probability of observing outcome $\eta^k_S$. For each $ k \in [N_S]$, we take $M$ independent batches, each of size $N_2$, of conditional samples of $\xi$ given $\rveta_S = \eta_S^k$, say $\xi^{kij}$ for $i \in [M], j \in [N_2]$. For each $k \in [N_1]$ and each batch $i \in [M]$, let
\begin{equation*}
\mathbf{\hat{R}}_{k,i,N_2}(\eta_S^k) = \max_{y \in Y} \beta^{\top} y + N_2^{-1} \sum_{j \in [N_2]} Q(y, \xi^{kij})
\end{equation*}
and let $\mu_{k,N_2}(\eta_S^k) = \mathbb{E}[\mathbf{\hat{R}}_{k,i,N_2}]$ be the expected value of each batch.
Then, for each $k \in [N_1]$ define the average of the batch values as
\begin{equation*}
\mathbf{\bar{R}}_{k,M,N_2}(\eta_S^k) = M^{-1} \sum_{i \in [M]} \mathbf{\hat{R}}_{k,i,N_2}(\eta_S^k).
\end{equation*}
The estimate $\mathbf{\bar{R}}_{k,M,N_2}(\eta_S^k)$ is an unbiased estimator of $\mu_{kN_2}(\eta_S^k)$. When the $M$ batch samples are i.i.d., by the Central Limit Theorem, we have that as $M\rightarrow \infty$
\begin{equation}
\label{eq:R_CLT}
\sqrt{M}[\mathbf{\bar{R}}_{k,M,N_2}(\eta_S^k) - \mu_{k,N_2}(\eta_S^k)]] \rightarrow N\Bigl(0,\operatorname{Var}\bigl(\mathbf{\hat{R}}_{k,i,N_2}(\eta_S^k) \bigr)\Bigr).
\end{equation}
The sample variance of $\mathbf{\hat{R}}_{k,i,N_2}(\eta_S^k)$ is computed as
\begin{equation*}
\mathbf{s}^2_{R_k}(M) = (M-1)^{-1} \sum_{i \in [M]}\bigl(\mathbf{\hat{R}}_{k,i,N_2}(\eta_S^k) - \mathbf{\bar{R}}_{k,M,N_2}(\eta_S^k)\bigr)^2.
\end{equation*}
Define
\begin{align}
\bar{F}(S) & := \sum_{k \in [N_S]} p(\eta_S^k) \Expect[ \mathbf{\bar{R}}_{N_2}(\eta_S^k)] \ \text{and} \nonumber \\ 
\mathbf{F}_{M,N_2}(S) & := \sum_{k \in [N_S]} p(\eta_S^k) \mathbf{\bar{R}}_{k,M,N_2}(\eta_S^k) \label{eq:internal_est_enumerate}.
\end{align}
Observe that
\begin{align*}
 \bar{F}(S)
& \geq \sum_{k \in [N_S]} p(\eta_S^k) R(\eta_S^k) = F(S).
\end{align*}
Therefore, $\bar{F}(S)$ is an upper bound of $F(S)$, and $\mathbf{F}_{M,N_2}(S)$ is an unbiased estimator of $\bar{F}(S)$ and we can use it as a statistical upper bound of $F(S)$. Aggregating \eqref{eq:R_CLT} for $k\in [N_S]$ yields
\begin{align*}
\sqrt{M}[\mathbf{F}_{M,N_2}(S) - \bar{F}(S)] \rightarrow N\Bigl(0,\sum_{k \in [N_S]} p(\eta_S^k)^2 \operatorname{Var}\bigl(\mathbf{\hat{R}}_{k,i,N_2}(\eta_S^k) \bigr)\Bigr)
\end{align*}
as $M \rightarrow \infty$.
Thus, for $M$ sufficiently large, we can approximate the probability that $F(S)$ exceeds a constant $u$ with the probability that a normal random variable with mean $\mathbf{F}_{M,N_2}(S)$ and standard deviation $\mathbf{s}_M(S)$ exceeds $u$, where
\begin{equation*}
\mathbf{s}_M(S) = \Bigl(\sum_{k \in [N_S]} p(\eta_S^k)^2 \mathbf{s}^2_{R_k}(M)\Bigr)^{1/2} .
\end{equation*}







\subsubsection{Estimating a Global Upper Bound}

Now consider a branch-and-bound tree with set of leaf nodes $\mathcal{P}$, where for each node $P \in \mathcal{P}$ the information-based upper bound is calculated as $c^{P} + F([n]\setminus S_0^P)$, where the constant $c^{P}$ is determined via either Lemma \ref{lem:svbr} or Lemma \ref{lem:subsetbr}, depending on the branching mechanism used. Thus, the best upper bound on $\zpesp$ that can be obtained from this tree is the value:
\[ \max \{ c^{P} + F([n] \setminus S_0^P) : P \in \mathcal{P} \}. \]

Using the upper bound procedure described in Section \ref{sss:sub}, we obtain an upper bound estimate $\mathbf{U}_P$ of $F([n]\setminus S_0^P)$ for each $P \in \mathcal{P}$ which is approximately normally distributed with an estimated mean, say $\mu_P$, and estimated standard deviation, say $s_P$. Then, a value $u$ is a global upper bound if and only if $\mathbf{U}_P \leq u$ for all $P \in \mathcal{P}$, and hence, using independence of the estimates for each node $P \in \mathcal{P}$, we can compute a $1-\alpha$ confident upper bound on $\zpesp$ as
\begin{equation*}
    u(\mathcal{P}) = \inf_{u \in \mathbf{R}} \Bigl\{ u :  \prod_{P \in \mathcal{P}} \mathbb{P}(\mathbf{U}_P \leq u) \geq 1-\alpha \Bigr\}.
\end{equation*}
Using the normal approximations, $u(\mathcal{P})$ can be found via
\begin{equation}
\label{eq:internal_ci_ub}
    u(\mathcal{P}) = \inf_{u \in \mathbf{R}} \Bigl\{ u :  \prod_{P \in \mathcal{P}} \Phi\Bigl( (u - \mu_P)/s_P \Bigr) \geq 1-\alpha \Bigr\},
\end{equation}
which can be estimated by binary search.

In our description of the branch-and-bound algorithm using internal sampling we assume that for each node we use a fixed sample to estimate the upper bound, so that the estimated upper bound at a node is only improved by branching on that node. While we do not explore this here since our focus in this paper is on deriving upper bound and branching techniques for the probing structure, we mention that one can also consider variations where the upper bound at leaf nodes is improved by doing more sampling, and the sampling effort is strategically allocated to leaf nodes -- see, e.g.,  \cite{norkin1998branch,xu2013empirical}.



\subsection{Statistical Lower Bounds on $F(S)$ and $z_{PESP}$}
\label{sss:slb}



Any candidate set of probing actions $S \subseteq [\xd]$ defines a feasible solution of \eqref{eq:mainprob}, and hence
\[ \zpesp \geq F(S)-\alpha(S). \]
Thus, a statistical lower bound on $\zpesp$ can be obtained by estimating a statistical lower bound on $F(S)$ for any $S$.

Let $\eta_S^{k}, k \in [N_1]$ be a sample of $\rveta_S$, and for each observation $\eta_S^k$,  choose a solution $\hat{y}^k \in Y$.  
For each $k \in [N_2]$, let $\hat{\xi}^{ki}, i \in [N_2]$ be a sample of $\rvxi$ taken from the conditional distribution of $\rvxi$ given $\rveta_S = \eta_S^k$. 
Then, an estimate of $F(S)$ is obtained as
\begin{equation}
\label{eq:lower_bound}
\mathbf{\underline{F}}_{N_1,N_2}(S) = N_1^{-1} \sum_{k \in [N_1]} \mathbf{\underline{R}}_{N_2}(\eta_S^k)
\end{equation}
where
\[ \mathbf{\underline{R}}_{N_2}(\eta_S^k) = \beta^{\top} \hat{y}^k + N_2^{-1} \sum_{i \in [N_2]} Q(\hat{y}^k, \hat{\xi}^{ki}) . \]
Computing this estimate requires solving $N_1 \cdot N_2$ recourse problems to evaluate $Q(\hat{y}^k,\hat{\xi}^{ki})$   for each $k \in [N_1]$ and $i \in [N_2]$.

As shown in Lemma \ref{lem:fslb}, $\mathbf{\underline{F}}_{N_1,N_2}(S)$ provides a statistical lower bound regardless of the choice of $\hat{y}^k \in Y$ for $k \in [N_1]$. However, the quality of this lower bound will be influenced by how good the solution $\hat{y}^k$ is to the problem defined by $R(\eta_S^k)$. Thus, we suggest to choose $\hat{y}^k$ by taking a separate conditional sample of $\rvxi$ given $\rveta_S = \eta_S^k$,  $\xi^{ik}, i \in [N_3]$ for each $k \in [N_1]$ and
choosing
\begin{equation}
\label{eq:ysel}
\hat{y}^k \in \arg \max_{y \in Y} \beta^{\top} y + N_3^{-1} \sum_{i \in [N_3]} Q(y, \xi^{ki}) . 
\end{equation}
Computing $\hat{y}^k$ for $k \in N_1$ in this way requires solving $N_1$ two-stage stochastic programs, each having $N_3$ scenarios.

\begin{lemma}
\label{lem:fslb}
Let $\hat{y}^k \in Y$ for $k \in [N_1]$.
Assume that $\eta^1,\ldots,\eta^{N_1}$ are identically distributed and for each $k \in [N_2]$ the sample $\xi^{k1},\ldots,\xi^{kN_2}$ is identically distributed conditional on $\rveta_S = \eta_S^k$. Then,
\[ \Expect_{\eta^{N_1},\xi^{N_2}}[ \mathbf{\underline{F}}_{N_1,N_2}(S) ] \leq F(S). \]
\end{lemma}
\begin{proof}
For each $k \in [N_1]$, $\hat{y}^k \in Y$ implies that 
\[ R(\eta_S^k) \geq \beta^{\top} \hat{y}^k + \mathbb{E}_{\rvxi} \bigl[ Q(\hat{y}^k, \hat{\xi}) \ \big| \rveta_S = \eta_S^k \bigr]. \]
Thus,
\begin{align*}
\Expect_{\eta^{N_1},\xi^{N_2}} \bigl[ \mathbf{\underline{F}}_{N_1,N_2}(S) \bigr] &= \Expect_{\eta^{N_1}} \Bigl[ N_1^{-1} \sum_{k \in [N_1]} \Expect_{\xi^{N_2}}[\mathbf{\underline{R}}_{N_2}(\eta_S^k)] \Bigr] \\
&= \Expect_{\eta^{N_1}} \Bigl[ N_1^{-1} \sum_{k \in [N_1]} \Expect_{\xi^{N_2}} \bigl[ \beta^{\top} \hat{y}^k + N_2^{-1} \sum_{i \in [N_2]} Q(\hat{y}^k, \hat{\xi}^{ki}) \bigr] \Bigr] \\
&= \Expect_{\eta^{N_1}} \Bigl[ N_1^{-1} \sum_{k \in [N_1]} \Bigl(\beta^{\top} \hat{y}^k + \Expect_{\rvxi}  \bigl[ Q(\hat{y}^k, \rvxi) | \rveta_S = \eta_S^k \bigr] \Bigr) \Bigr] \\
&\leq \Expect_{\eta^{N_1}} \Bigl[N_1^{-1} \sum_{k \in [N_1]}  R(\eta_S^k) \Bigr]  = \Expect_{\rveta_S} [ R(\rveta_S)] =  F(S). 
\end{align*}
\qed
\end{proof}

If the sample $\eta^1,\ldots\eta^{N_1}$ is {\it independent} and identically distributed, then the values $\mathbf{\underline{R}}_{N_2}(\eta_S^k)$ for $k \in [N_2]$ are iid and hence can be used to compute an approximate statistical lower bound for $F(S)$.
By the Central Limit Theorem, we have
\begin{equation*}
\sqrt{N_1}\Bigl(\mathbf{\underline{F}}_{N_1,N_2}(S) - \mathbb{E}\bigl[ \mathbf{\underline{R}}_{N_2}(\rveta_S)\bigr] \Bigr) \rightarrow N\Bigl(0,\operatorname{Var}\bigl(\mathbf{\underline{R}}_{N_2}(\rveta_S)\bigr)\Bigr)
\end{equation*}
as $N_1 \rightarrow \infty$.
The sample variance of $\underline{R}_{N_2}(\rveta_S)$ is computed as
\begin{equation*}
\mathbf{s}^2_{\underline{R}_{N_2}} = (N_1-1)^{-1} \sum_{k \in [N_1]} (\mathbf{\underline{R}}_{N_2}(\eta_S^k) - \mathbf{\underline{F}}_{N_1,N_2}(S))^2 .
\end{equation*}
Let $t_{N_1-1,\alpha}$ be the critical value of Student's t distribution.
Then, we obtain the $1-\alpha$ confidence lower bound estimate of $F(S)$ as
\begin{equation*}
\mathbf{\underline{F}}_{N_1,N_2}(S) - t_{N_1-1,\alpha} \mathbf{s}_{\underline{R}_{N_2}}N_1^{-1/2}.
\end{equation*}

\section{Greedy Heuristic}
\label{sec:heuristic}
At every node in the branch-and-bound method of Section \ref{sec:Branch_and_Bound} there is a natural candidate solution that can be used to obtain a valid (statistical) lower bound on the optimal solution value via the method described in Section~\ref{sss:slb}.  Moreover, in this case, there is no need to solve the auxiliary two-stage stochastic programs described in~\eqref{eq:ysel}, since the evaluation of $F(S)$ naturally provides high-quality candidate solutions $y^k$.
However, the solution is obtained by paying for probing for \emph{all} components of $\rveta$ that are not fixed at the node.  Hence, especially at the top of the branch-and-bound tree, the feasible solutions probe many components of $\rveta$, and the solutions may be far from optimal. In this section, we describe a greedy heuristic to obtain a feasible solution to~\eqref{eq:mainprob}. While the greedy method is very natural, our contribution is to identify techniques for relatively efficiently evaluating candidate probing actions to greedily add to the current solution at each iteration.

We first state in Algorithm \ref{alg:greedy} the greedy heuristic in its natural form, which is not practical to implement due to the need to evaluate (or even estimate) $F(S)$ for a large number of sets $S$. We use the notation $S^c = [n] \setminus S$. The method begins with $S=\emptyset$ as the initial solution.  In each iteration we have a solution $S$, and for each element $j \in S^c$, we evaluate the solution obtained by adding that element to $S$, and choose one that gives the largest value. As it is possible that the solutions obtained may improve or decline at each iteration, the method continues until all elements are probed and stores all selected solutions in the set $\mathcal{L}$, at the end returning the best solution obtained.

\begin{algorithm}
    \caption{Naive greedy heuristic.}\label{alg:greedy}
    $S \gets \emptyset$. $\mathcal{L} \gets \{S\}$\\
    \Repeat{$S = [n]$}{
        Evaluate $F(S)$. \label{evalfs} \\
        \For{$j \in S^c$} {
            $z_j \gets \alpha(S \cup \{j\}) + F(S \cup \{j\})$ \label{evalfnb}
        }
        Choose $j^* \in \arg \max\{ z_j : j \in S^c\}$ \\
        $S \gets S \cup \{j^*\}$, $\mathcal{L} \gets \mathcal{L} \cup S$.
    }
    \Return $\arg \max\{  F(S) -\alpha(S) : S \in \mathcal{L} \}$.
\end{algorithm}

The naive greedy algorithm requires $O(n^2)$ evaluations of $F(S)$, which is impractical unless the joint support of $(\rveta,\rvxi)$ is small.  We thus employ the conditional sampling method of Section \ref{sec:intsamp} to estimate $F(S)$. However, even doing this estimation $O(n^2)$ times is computationally prohibitive for a heuristic. 

In our proposed heuristic, each time we obtain a new solution $S$, the evaluation of $F(S)$ (line \ref{evalfs} of the naive greedy heuristic) is conducted using a slight adaptation of the sampling procedure for estimating a lower bound described in Section \ref{sss:slb}. Just as in Section \ref{sss:slb}, we first take a sample $\eta_S^1,\ldots,\eta_S^{N_1}$ of $\rveta_S$, and for each $k \in [N_1]$ we obtain a solution $\hat{y}^k \in Y$ by solving \eqref{eq:ysel} using a conditional sample  $\xi^{k1},\ldots,\xi^{kN_3}$ of $\rvxi$ conditional on $\rveta_S = \eta_S^k$. At this point, the evaluation diverges slightly in that we next take a {\it joint sample} $(\eta_{S^c}^{k1},\xi^{k1}),\ldots,(\eta_{S^c}^{kN_2},\xi^{kN_2})$ of $(\rveta_{S^c},\rvxi)$ conditional on $\rveta_S = \eta_S^k$ for each $k \in [N_1]$. This joint sample will be used to 
facilitate a faster estimation of $F(S \cup \{j\})$ for each $j \in S^c$ when evaluating the next element to add to the set in the greedy method. Next, we let $\hat{Y} = \{\hat{y}^1,\ldots, \hat{y}^{N_1}\}$ be the observed solutions and compute the values $Q(\hat{y}^k,\xi^{ki})$ for all $k \in [N_1],$ and $i \in [N_2]$, which requires solving $N_1 \times N_2$ recourse problems.
Then, the heuristic estimates $F(S)$ by
\begin{equation}
\label{eq:approxfs}
 N_1^{-1} \sum_{k \in [N_1]} \max_{y \in \hat{Y}} \Bigl\{ \beta^{\top} y + N_2^{-1} \sum_{i \in [N_2]} Q(y, \hat{\xi}^{ki}) \Bigr\}. 
 \end{equation}
Note that the set $Y$ is replaced by the limited set of solutions $\hat{Y}$, so that this quantity can be computed without solving any additional optimization problems. 

An important aspect of our heuristic is the ability to reuse computations to quickly estimate $F(S \cup \{j\}$ for $j \in S^c$.  
We next fix $j \in S^c$ and discuss how we estimate $F(S\cup\{j\})$. For each $k \in [N_1]$, we use $K$-means clustering  to partition the (scalar) values $\{\eta^{ki}_j: i \in [N_2]\}$ into $K$ sets of similar values, and let the scenarios in these sets be $\Omega_{N_2}^{k\ell}$ for $\ell \in [K]$, so that these sets are disjoint and $\bigcup_{\ell \in [K]} \Omega_{N_2}^{k\ell} = [N_2]$. Our approximation is based on the assumption that taking probing action $j$ will allow us to distinguish between scenarios in different sets $\Omega_{N_2}^{k\ell}$, but not scenarios within each of these sets. We also continue to use the set $\hat{Y}$ in place of $Y$, and thus estimate the value of $F(S\cup\{j\})$ as
\begin{equation}
\label{eq:approxzj}
 N_1^{-1} \sum_{k \in [N_1]} K^{-1} \sum_{\ell \in [K]} \max_{y \in \hat{Y}} \Bigl\{ \beta^{\top} y + |\Omega_{N_2}^{k\ell}|^{-1} \sum_{i \in \Omega_{N_2}^{k\ell}} Q(y, \hat{\xi}^{ki}) \Bigr\} . 
 \end{equation}
The key observation is that $Q(y,\xi^{ki})$ has already been computed for all $y \in \hat{Y}$, $k \in [N_1]$, and $i \in [N_2]$, and hence computing \eqref{eq:approxzj} does not require solving any additional optimization problems. 

Thus, our proposed greedy heuristic follows the structure of Algorithm \ref{alg:greedy}, with the differences being that the evaluation of $F(S)$ in line \ref{evalfs} is replaced by the approximation \eqref{eq:approxfs} and the evaluation of $F(S \cup \{j\})$ in line \ref{evalfnb} is replaced by \eqref{eq:approxzj}.

The computational effort of the greedy heuristic is impacted heavily by the choice of the sample sizes $N_1$ and $N_2$. To generate solutions relatively quickly, one may use relatively small sample sizes -- e.g., we use $N_1 = 20$, $N_2 = 20$, and $N_3=50$ in our experiments. While smaller sample sizes may be sufficient for guiding the greedy search, it is important to evaluate the most promising solutions using larger samples to estimate a lower bound as described in Section \ref{sss:slb}.  In our implementation, we use the estimated solution values obtained by \eqref{eq:approxfs} within the heuristic to select the ten most promising solutions from the set of solutions $\mathcal{L}$, and then evaluate these selected solutions with larger sample sizes. 

\section{Computational Results}
\label{sec:comp}

In this section, we first introduce a prototype application. Then we present the results of a series of computational experiments that demonstrate the impact of various algorithmic choices and the effectiveness of our branch-and-bound algorithm.

\newcommand{\confcost}{\beta^{\text{\small c}}}
\newcommand{\assigncost}{\beta^{\text{\small a}}}
\newcommand{\revenue}{r}

\subsection{Probing-Enhanced Facility Location}
\label{sec:New_Instance}
Our sample application is an extension of a two-stage stochastic facility location problem.  There is a set of potential facilities $I$, and each of the facilities $i \in I$ may be built in one of a given set of (capacity) configurations $C_i$.  There are also a given set of customers $J$, and each customer $j \in J$ has a (random) demand of $\bfd_j$.  Customers may only have their demand served from a single facility.  The objective of the problem is to select a set of facilities to open, to configure the facilities, and to assign customers to facilities in order to maximize the expected profit.  There is a (fixed) cost $\confcost_{ic}$ of opening facility $i \in I$ in configuration $c \in C_i$ and a fixed cost $\assigncost_{ij}$ of assigning customer $j \in J$ to facility $i \in I$.  These assignments must be done before observing the customer demands.  There are binary variables $y_{ic}$ that take the value 1 if and only if facility $i$ is opened in configuration $c \in C_i$ and binary variables $u_{ij}$ that take value 1 if and only if customer $j \in J$ is assigned to facility $i \in I$.  The stochastic facility location problem can then be written as 
\begin{subequations}
\label{eq:sfl-1st}
\begin{alignat}{2}
\max_{y,u} \ & - \sum_{i \in I} \sum_{c \in C_i} \confcost_{ic} y_{ic} - \sum_{i \in I} \sum_{j \in J} \assigncost_{ij} u_{ij} &&+ \Expect_{\bfd} [Q(y,u,\bfd)] \label{eq:sfl-objective}\\ 
\mbox{s.t. } &  \sum_{c \in C_i} y_{ic}  \leq 1  &&\forall i \in I, \label{eq:2-stage config_constr} \\
& u_{ij} - \sum_{c \in C_i} y_{ic}  \leq 0  &&\forall i \in I, \forall j \in J,  \label{eq:2-stage assign_config_constr} \\
&  \sum_{i \in I} u_{ij} \leq 1\qquad &&\forall j \in J, \label{eq:2-stage assign_sum_constr} \\
& y_{ic} \in \{0,1\}  &&\forall i \in I, \forall c \in C_i,\\
& u_{ij} \in \{0,1\}  &&\forall i \in I, \forall j \in J .
\end{alignat}
\end{subequations}

Constraints~\eqref{eq:2-stage config_constr} and~\eqref{eq:2-stage assign_config_constr} ensure that at most one configuration is selected for each facility and that customers are assigned only to open facilities. Constraint \eqref{eq:2-stage assign_sum_constr} ensures that customers are only assigned to one facility. 
In~\eqref{eq:sfl-objective}, the term $\Expect_{\bfd} Q(y,u,\bfd)$ is the expected revenue obtained when opening and sizing facilities according to $y$ and assigning customers to facilities according to $u$.  For a given realization of demands $\bfd(\omega) = d$ and first-stage solution $y$ and $u$, the maximum revenue can be obtained by solving the following linear program: 
\begin{subequations}
\label{eq:sfl-2nd}
\begin{alignat}{2}
Q(y,u,d) := \max_{f} \quad & \sum_{i \in I} \sum_{j \in J} \revenue f_{ij} \label{eq:2stage-obj} \\
\mbox{s.t. } & \sum_{j \in J} f_{ij} - \sum_{c \in C_i} \theta_{ic} y_{ic}  \leq 0 \qquad &&\forall i \in I, \label{eq:2-stage flow_config_constr} \\
 & f_{ij} - d_j u_{ij} \leq 0 \qquad &&\forall i \in I, \forall j \in J, \label{eq:2-stage flow_demand_constr} \\
 & f_{ij}  \geq 0 \qquad &&\forall i \in J, \forall j \in J.
\end{alignat}
\end{subequations}
The decision variable $f_{ij}$ is the number of units of product that customer $j \in J$ received from facility $i$, and each unit of customer demand that is met results in a revenue of $\revenue$.  The parameters $\theta_{ic}$ are the production capacity of facility $i$ if operated in configuration $c \in C_i$.

\subsubsection{Definition of Uncertainty}
\label{sec:def_uncertainty}
In our computational experiments, we have families of instances with two different distributions of customer demands.  In the first instances, the customer demands are independent random variables following a discrete distribution with two outcomes---either the customer demand is $0$ with probability $\rho_j$, or it is a known nominal value $\eta_j$ with probability $(1-\rho_j)$.  Thus, instances in this family have a total of $2^{|J|}$ scenarios. 

In the second family, customer demands are continuous random variables. Each customer $j$ can either have a low demand or a high demand. The probability of a low demand for customer $j$ is $\rho_j$, and the probability of a high demand is $(1 - \rho_j)$. For a low demand, the demand follows a triangular distribution with a minimum value of 0, a maximum value of $\eta_1$, and a mode of 0. For a high demand, the demand follows a triangular distribution with a minimum value of $\eta_2$, mode $\eta_3$, and maximum value of $\eta_4$.  The demand for each customer is independent.  Instances are available from the authors in JSON format upon request.


\subsubsection{Information Model}

The probing-enhanced stochastic programming model introduction in Section~\ref{sec:Model} allows for an arbitrary correlation structure between $\rveta$ and $\rvxi$.  In this proof-of-concept implementation, we assume that $\rveta = \rvxi$.  In the context of our facility location model, this implies that by probing customer $j \in J$, we can exactly know that customer's true demand realization. However, because our stochastic model assumes the customer demands are independent, this gives no information about the distribution of demand of {\it other} customers.  This information structure is equivalent to earlier work on \emph{decision-dependent information structure} that models the ability to control the \emph{timing} of when the outcome of random variables is known.  Here, if we probe customer $j \in J$, then we know its demand before deciding facility locations, sizes, and customer assignment to open facilities.  Future work will consider different models of correlation  between $\rveta$ and $\rvxi$.

\subsection{Computing Environment and Test Instances}

All two-stage stochastic programs are solved via the extensive form  \citep{SPbook_birge1997} and these and any other optimization problems are solved with Gurobi v9.5.1.
A computational advantage of the branch-and-bound methods proposed in  Section~\ref{sec:Branch_and_Bound} when combined with sampling is that independent estimations of problem bounds can be computed in parallel. 
 Most of the computations in this section were carried out on a shared cluster of machines of varying architectures scheduled with the HTCondor software \citep{thain.tannenbaum.livny:05}. 
 Thus, in order to report computational effort and compare different approaches we rely primarily on Gurobi's \emph{work unit} (GWU), which in our experience is a more reliable statistic to use for comparison between runs on different machines. According to Gurobi's documentation, a GWU is very roughly equivalent to a second of CPU runtime, but this may vary significantly depending on the hardware.  Unless otherwise specified, we enforce a total computational budget of 30000 GWU for obtaining an upper bound using our branch-and-bound methods.  We also report other machine-independent statistics of computational effort, such as the total nodes of all branch-and-bound trees, the total number of MIP problems solved, etc.  For computations done in parallel, we report statistics aggregated over all distributed computations.

Most of the computational effort in the information-relaxation-based branch-and-bound method is on the computation or estimation of $F(S)$ for different subsets $S \subseteq [n]$ throughout the branch-and-bound tree. 
Recall from the description of the single-element branching method of Section~\ref{sec:single-element-branching} that the value of $F(S)$ needs to be re-computed for only one of the two child nodes.  Similarly, the value of $F(S)$ needs to be recomputed for only two of the three child nodes in the multi-element branching method of Section~\ref{sec:multi-element-branching}.  So in some computational tables, we report both the number of nodes in the tree and the number function evaluations of $F(S)$ required.


We do a majority of our testing on twelve instances, six with discrete demand distribution and six with continuous distribution as described in Section~\ref{sec:def_uncertainty}  All instances have $|I|=5$ facilities, and each facility has $|C_i| = 4$ possible configurations.  The name of each instance encodes the number of customers, with an instance whose name starts with ``J$n$'' having $n=|J|$ customers, and instances whose names end in ``\_C'' have customers whose demands follow a continuous distribution.  





\begin{table}[hbtp]
\begin{center}
\caption{Comparison of Nonanticipative Formulation and Information-Relaxation Based Branch-and-Bound Algorithm}
\label{tab:na-bad}
\begin{tabular}{c|rrr|rrr}
& \multicolumn{3}{c|}{{\bf NA-Formulation}} & \multicolumn{3}{c}{{\bf Info-Relax B\&B}}\\
{\bf Instance} & \multicolumn{1}{c}{\bf \# Nodes} & \multicolumn{1}{c}{\bf GWU} & \multicolumn{1}{c|}{\bf Gap (\%)} & \multicolumn{1}{c}{\bf \# Nodes} & \multicolumn{1}{c}{\bf \# Eval} & \multicolumn{1}{c}{\bf GWU}  \\ \hline
J4 & 1 & 1.5 & 0.0 & 11 & 8 & 0.1\\
J5 & 235 & 173.0 & 0.0 & 18 & 14 & 0.2\\
J6 & 170 & 5411.6 & 0.0 & 47 & 39 & 1.8\\
J7 & 1 & 30000.0 & 39.7 & 44 & 37 & 4.9\\
J8 & 1 & 30000.0 & 191.2 & 89 & 74 & 31.1\\
J9 & - & - & - & 101 & 81 & 84.4\\
J10 & - & - & - & 276 & 234 & 363.7\\ \hline 
J20\_1 & 1 & 30000.0 & 30.7 & 85920 & 75050 & 1801.7 \\ 
J20\_2 & 1 & 30000.0 & 35.2 & 40399 & 35234 & 652.3 \\ 
J20\_3 & 1 & 30000.0 & 7.4 & 72893 & 63304 & 586.5 \\ \hline
\end{tabular}
\end{center}
\end{table}

\subsection{Comparison of Nonanticipative and Nested Formulations}
We first report results of experiments designed to compare the performance of a standard branch-and-cut-based MIP solver on the nonanticipativity-based formulation~\eqref{eq:na-mip} with the information-relaxation-based branch-and-bound algorithm for solving the probing-enhanced facility location problem.  We perform two types of comparisons.  First, we exactly solve small instances having between 4 and 10 customers and a discrete demand distribution by both methods.  Second, we solve sampled instances (with sample size $N=100$) that have 20 customers by both methods.  The multi-way branching method describe in Section~\ref{sec:multi-element-branching} is used for the information-relaxation-based scheme.  Table~\ref{tab:na-bad} shows the results of these experiments, where \# Nodes is the number of nodes explored either within the work unit limit or until terminated, GWU is the Gurobi work units, Gap (\%) is the ending optimality gap reported by the MIP solver when the work limit was reached, and \# Eval is the number of nodes in the information-based branch-and-bound algorithm on which the $F(S)$ needed to be evaluated.


For instances J9 and J10, the nonanticapative formulation could not be created due to a lack of memory.  The information-relaxation-based branch-and-bound method was able to solve all instances to optimality within the work-unit limit.  The table indicates that the information-relaxation-based branch-and-bound method can outperform the direction solution of the nonanticipativity-based formulation by several orders of magnitude.

\subsection{Impact of Variance Reduction in Sampling Methods}
\label{sec:results-sampling-method}

The quality of the statistical estimates of the solution bounds described in Section~\ref{sec:Sampling} depend heavily on the variance of observations.  In Table~\ref{tab:lhs_good}, we demonstrate the significant reduction in variance in the estimates that can be obtained by using Latin hypercube sampling (LHS).  The table shows for the instance J20\_3 the estimate of $\mathbb{E}[\bfv_N]$ for two different values of $N$ and the standard deviation of the estimate.  There were $L=30$ replications used to to compute the estimates.
The estimate $\Upsilon_{N,L}$ was computed by equation~\eqref{eq:saa_avg_val}, and the standard deviation of the estimate of the mean $s_{N,L}$, was computed by equation~\eqref{eq:saa_var_val}.
The table clearly shows a significant variance reduction when sampling via LHS for these instances, so all remaining computational results presented use LHS for obtaining estimates. 

\begin{table}[hbtp]
\centering
\caption{Sampling Method Test on Instance J20\_3}
\label{tab:lhs_good}
\begin{tabular}{c|r|rr}
\toprule
{\bf Sampling} & \multicolumn{1}{c}{\bf $N$} & \multicolumn{1}{c}{\bf $\Upsilon_{N,L}$} & \multicolumn{1}{c}{\bf $s_{N,L}$} \\ \hline
MC & 50 & 12431.8 & 964.9 \\
LHS & 50 & 12527.8 & 110.9 \\ \hline
MC & 100 & 12464.6 & 540.2 \\
LHS & 100 & 12368.0 & 78.7 \\
\bottomrule
\end{tabular}
\end{table}

\subsection{Impact of Storing $R_N(\seta_S)$ Evaluations}
\label{sec:hashing_good}

Recall from the description of using the branch-and-bound method to solve an externally-sampled instance in Section~\ref{sec:extsamp} that the value of the same two-stage stochastic program computed as $R_N(\seta_S)$ in~\eqref{eq:rncalc_saa} may be required at many different nodes in the branch-and-bound tree.
In Table~\ref{tab:hashing_good}, we show how often the values are reused 
for our discrete instances with $|J|=20$ customers, externally sampled with a sample size of $N=50$, and replicated $L=30$ times.  Our implementation uses a hash table which (for memory purposes) stores only the most recently used $10,000$ values of $R_N(\seta_S)$. For these instances, we show the average GWU, 
the average number of nodes in the branch-and-bound trees (\# Nodes), the average number of function evaluations of $F(S)$ that are required (\# Eval.), the average number of two-stage stochastic programs (values $R_N(\seta_S)$) that are required for the computation (\#SP), and the percentage of time the value of $R_N(\seta_S)$ can be recovered from a stored value (\% Lookup).  We show this data for both the single-variable branching method of Section~\ref{sec:single-element-branching} and multi-variable branching of Section~\ref{sec:multi-element-branching}.


\begin{table}[hbtp]
\centering
\caption{Impact of Hashing}
\label{tab:hashing_good}
\begin{tabular}{c|c|rrrrr}
\toprule
{\bf Instance} & {\bf Method} & \multicolumn{1}{c}{\bf GWU} & \multicolumn{1}{c}{\bf \# Nodes} & \multicolumn{1}{c}{\bf \# Eval.} & \multicolumn{1}{c}{\bf \# SP} & \multicolumn{1}{c}{\bf \% Lookup} \\ \hline
J20\_1 & single & 110.8 & 116332.6 & 83017.7 & 3531655.8 & 99.79\\
J20\_1 & multi & 112.4 & 32770.9 & 29181.9 & 1174542.4 & 99.35\\ \hline
J20\_2 & single & 111.9 & 78279.1 & 56509.7 & 2493461.7 & 99.81\\
J20\_2 & multi & 117.3 & 23871.6 & 21187.1 & 895986.7 & 99.44\\ \hline
J20\_3 & single & 52.2 & 82441.4 & 58852.5 & 2585136.0 & 99.83\\
J20\_3 & multi & 65.4 & 29537.0 & 26030.0 & 1098195.1 & 99.50\\
\bottomrule
\end{tabular}
\end{table}

Table~\ref{tab:hashing_good} shows that a remarkably high percentage of the values $R_N(\seta_S)$ can be reused during the course of the branch-and-bound search. 
The percentage is uniformly larger for the single-variable branching method,  which is consistent with the intuition that the child node subproblems change less under single-variable branching, and if a subproblem changes less, then it is more likely to be able to reuse computations when re-evaluating its bound.


Figure~\ref{fig:hashing_good} displays this phenomenon graphically, depicting the evolution of the upper bound of the branch-and-bound tree as a function of the GWU for an externally-sampled ($N=100$) instance of J20\_1 for both single-variable and multi-variable branching rules. The figure demonstrates that the bound evolution is very similar when hashing is used, but significantly better for the multi-variable branching rule if the values of $R(\seta_S)$ are not stored and re-used. However, using hashing leads to significantly faster reduction of the upper bound in both cases.

\begin{figure*}[hbt]
    \centering
    \begin{subfigure}[t]{0.5\textwidth}
        \centering
        \includegraphics[width=\textwidth]{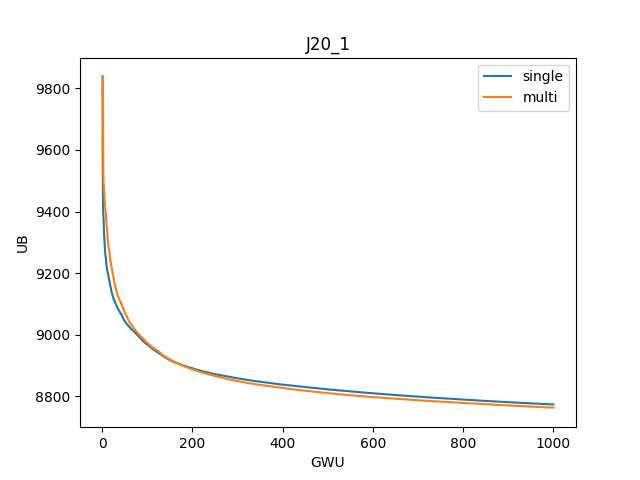}
        \caption{Use Hash Table to Look Up Values}
        \label{fig:j20_1-hash}
    \end{subfigure}%
    ~ 
    \begin{subfigure}[t]{0.5\textwidth}
        \centering
        \includegraphics[width=\textwidth]{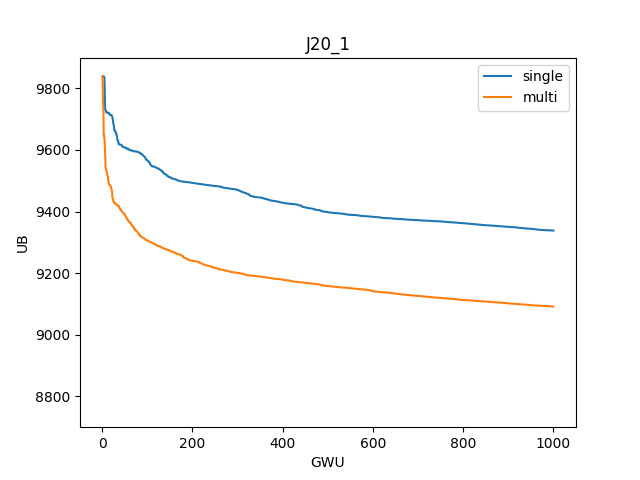}
        \caption{Recompute All Values}
        \label{fig:j20_1-nohash}
    \end{subfigure}
    \caption{Evolution of upper bound for externally-sampled instance of J20\_1 instance with $N=100$}
    \label{fig:hashing_good}
\end{figure*}

                





\subsection{Comparison of Branching Methods}
\label{sec:branch_compare}

This subsection is aimed at providing an empirical comparison between the single-element branching method of Section~\ref{sec:single-element-branching} and the multi-element branching method of Section~\ref{sec:multi-element-branching}.

\subsubsection*{External Sampling }

As explained at the beginning of Section~\ref{sec:intsamp}, employing the external sampling method for instances with a continuous probability distribution will yield only the perfect-information bound.  Thus, to evaluate branching methods for external sampling, we consider only our test instances that come from a discrete probability distribution.

In the experiment, we run the 
information-relaxation-based branch-and-bound method of Section~\ref{sec:Branch_and_Bound} on sampled instances of sizes $N \in \{50,100,200\}$ with different branching rules and compare the performance.  We employ $L=30$ different joint samples of $(\rveta,\rvxi)$ for each value of $N$ to estimate $\mathbb{E}[\bfv_N]$ via equation~\eqref{eq:saa_avg_val}. Each replication is run with a GWU limit of 1000. 

We test three branching methods in this experiment. Method {\tt random} follows the single variable branching method presented in Section~\ref{sec:single-element-branching} with the choice of probing element to branch on chosen randomly among the candidates.
Method {\tt single} also does single variable branching, but uses the first option described in the 
\emph{branching variable selection} portion of Section~\ref{sec:single-element-branching} for selecting the probing element to branch on.  Specifically, we use the observed difference between values of $R(\rveta_S)$ for the two outcomes for probing customer $j$, 
($\Expect_{\rveta_S} [R(\rveta_{S})|\rveta_j = 1] - \Expect_{\rveta_S} [R(\rveta_{S})|\rveta_j = 0]$), for the measure of importance of excluding $j \in [n]$ from probing consideration,  $\hat{\Delta}_j^-$. Method {\tt multi} follows the multi-element branching scheme described in Section~\ref{sec:multi-element-branching}, with branching decisions made using the same scoring as in method {\tt single}.

Table~\ref{tab:branching_compare_saa} shows the results of this experiment.
This table displays, for each combination of instance, sample size ($N$), and branching method employed: the number of the 30 replications that are successfully solved within the work unit limit (Solved), the average number of nodes solved over the $L=30$ replications (\#Nodes), the average number of evaluations of the function $F(S)$ that are required (\#Eval), the average number of work units (GWU), and the estimate of the upper bound $\mathbf{\Upsilon}_{N,L}$ (UB). We do not test method {\tt rand} on the instances with $N=200$ because the instances with $N \in \{50,100\}$ were sufficient to demonstrate the superior performance of {\tt single} over {\tt rand}.

Since for each of the three branching methods, all $L=30$ of the sampled instances having $N=50$ solved successfully, the estimated value of $\mathbb{E} [{\bfv}_{50}]$ (in UB) is the same in all cases.  In case that one or more of the 30 replications did not solve, we use upper bound of the branch-and-bound tree at the work limit as the value $\bfv^\ell_N$ in equation~\eqref{eq:saa_avg_val} to obtain a valid statistical upper bound on $\zpesp$.

\begin{table}[hbtp]
\centering
\caption{Branching Method Performance for Externally Sampled Instances}
\label{tab:branching_compare_saa}
\begin{tabular}{ccc|rrrr|r}
\toprule
{\bf Instance} & {\bf $N$} & {\bf Method} & \multicolumn{1}{c}{\bf Solved} & \multicolumn{1}{c}{\bf \# Nodes} & \multicolumn{1}{c}{\bf \# Eval.} & \multicolumn{1}{c}{\bf GWU} & \multicolumn{1}{c}{\bf 95\% UB} \\ \hline
J20\_1 & 50 & random & 30 & 124479.0 & 89333.8 & 256.8 & 9012.6\\
J20\_1 & 50 & single & 30 & 116332.6 & 83017.7 & 110.8 & 9012.6\\
J20\_1 & 50 & multi & 30 & 32770.9 & 29181.9 & 112.4 & 9012.6\\ \hline
J20\_2 & 50 & random & 30 & 86553.5 & 62983.4 & 212.0 & 10785.5\\
J20\_2 & 50 & single & 30 & 78279.1 & 56509.7 & 111.9 & 10785.5\\
J20\_2 & 50 & multi & 30 & 23871.6 & 21187.1 & 117.3 & 10785.5\\ \hline
J20\_3 & 50 & random & 30 & 99418.4 & 71926.8 & 121.1 & 12563.9\\
J20\_3 & 50 & single & 30 & 82441.4 & 58852.5 & 52.2 & 12563.9\\
J20\_3 & 50 & multi & 30 & 29537.0 & 26030.0 & 65.4 & 12563.9\\ \hline
J20\_1 & 100 & random & 0 & 78585.1 & 55584.7 & 1000.1 & 8999.1\\
J20\_1 & 100 & single & 0 & 162637.2 & 112977.8 & 1000.1 & 8890.3\\
J20\_1 & 100 & multi & 0 & 63584.7 & 55793.2 & 1000.0 & 8865.7\\ \hline
J20\_2 & 100 & random & 0 & 74902.0 & 53348.6 & 1000.0 & 10726.5\\
J20\_2 & 100 & single & 8 & 147508.7 & 102501.2 & 937.6 & 10632.6\\
J20\_2 & 100 & multi & 10 & 54807.4 & 47906.7 & 943.9 & 10627.2\\ \hline
J20\_3 & 100 & random & 0 & 119891.7 & 84635.9 & 1000.0 & 12425.7\\
J20\_3 & 100 & single & 25 & 161778.2 & 111920.8 & 645.5 & 12358.5\\
J20\_3 & 100 & multi & 28 & 68844.8 & 59858.6 & 668.8 & 12356.2\\ \hline
J20\_1 & 200 & single & 0 & 59605.0 & 41213.0 & 1000.0 & 8934.9\\
J20\_1 & 200 & multi & 0 & 16636.9 & 14489.3 & 1000.0 & 8891.9\\ \hline
J20\_2 & 200 & single & 0 & 52628.9 & 36447.0 & 1000.0 & 10689.6\\
J20\_2 & 200 & multi & 0 & 14464.6 & 12522.2 & 1000.0 & 10671.5\\ \hline
J20\_3 & 200 & single & 0 & 80465.0 & 55133.2 & 1000.0 & 12364.1\\
J20\_3 & 200 & multi & 0 & 25446.3 & 21803.1 & 1000.0 & 12373.0\\ \bottomrule
\end{tabular}
\end{table}

From Table~\ref{tab:branching_compare_saa} we conclude from comparing the {\tt random} and {\tt single} branching methods  that our strategy for selecting the branching entity described significantly outperforms branching randomly.  
We next observe that there is little difference in performance between single and multi-variable branching in terms of the total number of work units required to solve an instance or the bound obtained in a fixed number (1000) of work units.  Note, however, that the number of nodes and function evaluations required are significantly larger for single-variable compared to multi-variable branching. To explain this apparent contradiction, the reader is reminded of the results of Section~\ref{sec:hashing_good}, where it is demonstrated that single variable branching allows for significantly more reuse of portions of the $F(S)$ computations.  A second, more subtle reason for faster evaluation of nodes in single-variable branching is that the number of excluded candidate probing actions at nodes by multiple-variable branching is larger than for single-node branching.  As mentioned in Section~\ref{sec:pesp}, the difficulty of evaluating $F(S)$ depends inversely on the cardinality of $S$, so as more candidates probing actions are excluded from consideration, the upper bound calculation of $F([n] \setminus S_0^P)$ tends to require more computational effort for branch-and-bound nodes in a multiple-variable branching tree.

\begin{table}[hbtp]
\centering
\caption{Branching Method Performance for Internally Sampled Instances}
\label{tab:branching_internal}
\begin{tabular}{cc|rr|r}
\toprule
{\bf Instance} & {\bf Method} & \multicolumn{1}{c}{\bf \# Nodes} & \multicolumn{1}{c}{\bf \# Eval.} & \multicolumn{1}{c}{\bf 95\% UB} \\ \hline
J20\_1 & single & 844 & 547 & 8940.1\\
J20\_1 & multi & 275 & 218 & 8834.1\\ \hline
J20\_2 & single & 529 & 343 & 10607.7\\
J20\_2 & multi & 124 & 94 & 10611.5\\ \hline
J20\_3 & single & 1260 & 838 & 12076.8\\
J20\_3 & multi & 358 & 283 & 12283.4\\ \hline
J25\_1 & single & 847 & 539 & 9446.4\\
J25\_1 & multi & 214 & 165 & 9502.2\\ \hline
J25\_2 & single & 602 & 393 & 11091.0\\
J25\_2 & multi & 158 & 121 & 11172.5\\ \hline
J25\_3 & single & 342 & 225 & 10553.9\\
J25\_3 & multi & 178 & 135 & 10372.6\\ \hline
J20\_1\_C & single & 223 & 149 & 10496.5\\
J20\_1\_C & multi & 167 & 127 & 10401.4\\ \hline
J20\_2\_C & single & 183 & 122 & 10709.1\\
J20\_2\_C & multi & 127 & 98 & 10584.9\\ \hline
J20\_3\_C & single & 193 & 127 & 11423.6\\
J20\_3\_C & multi & 137 & 104 & 11279.8\\ \hline
J25\_1\_C & single & 141 & 93 & 11688.8\\
J25\_1\_C & multi & 86 & 64 & 11560.1\\ \hline
J25\_2\_C & single & 153 & 101 & 10147.7\\
J25\_2\_C & multi & 104 & 78 & 10064.7\\ \hline
J25\_3\_C & single & 119 & 81 & 9231.3\\
J25\_3\_C & multi & 68 & 51 & 9093.6\\
\bottomrule
\end{tabular}
\end{table}

\subsubsection*{Internal Sampling}

Table~\ref{tab:branching_internal} shows the results of an experiment comparing the single and multiple variable branch-and-bound methods on instances solved via our internal-sampling approach.  

Our internal sampling implementation uses LHS to generate the external sample of size $N_1 = 300$ as 30 independent batches of size 10.  The conditional sample has size $N_2 = 100$, and the statistical upper bound is estimated via equation~\eqref{eq:fubdef}.  If, however, the cardinality of the support of $\rveta_S$ is less than or equal to $8$, we forgo the sample and enumerate all possible outcomes, as explained at the end of subsection~\ref{sss:sub}.  In this case, for each possible $\eta_S^k$, we take $M=30$ independent batches of conditional samples $\xi$ given $\rveta_S = \eta_S^k$ of size $N_2=100$ and we estimate a statistical upper bound on $F(S)$ using equation~\eqref{eq:internal_est_enumerate}.  When selecting a branching candidate, we use the observed sample covariance between $\rveta_j$ and $R(\rveta_S)$ (from equation~\eqref{eq:branching_covariance}) for the importance score $\hat{\Delta}_j$.



All instances in the table were run with a maximum GWU limit of 30000.  The table shows the number of nodes of the branch-and-bound tree (\#Nodes), the number of evaluations of $F(S)$ (\# Eval), and the 95\% confident estimate of an upper bound on $\zpesp$ (95\% UB) computed by~\eqref{eq:internal_ci_ub} obtained by the method within the work limit by both branching methods.

For instances whose random variables are discrete, again the performance difference between single and multiple variable branching seems negligible in the internal sampling method.  
Note that single-variable branching is able to evaluate significantly many more nodes within the work limit than multiple-variable branching.  However, as the sampling is conditional in the internal-sampling approach, there is no ability to reuse portions of the $F(S)$ computations.  Thus, the reason for the similarity in performance between the methods is different 
for internal sampling than it is for external sampling.  
In the case of internal sampling, the difference is explained by the fact that in single-variable branching, nodes $P$ tend to have fewer candidate probing actions excluded, so the estimation of $F([n] \setminus S_0^P)$ is somewhat easier.  Also, there are more nodes whose bounds can be estimated using the ability to exploit the small support of the random variable $\rveta_S$, as explained at the end of 
Section~\ref{sss:sub}. 



For continuous distribution instances, the performance of the multiple-variable branching seems slightly better than the single-variable branching.

\subsection{Comparison of External and Internal Sampling}
\label{sec:comp-results-compare}

We next compare the performance of external and internal sampling methods in terms of the quality of upper bounds they are able to obtain within the same work limit.  
The runs (and all parameters values) used to make this comparison are the same as used for the experiments described in Section~\ref{sec:branch_compare}, and we use the multi-variable branching method. For the internal sampling method we use a single branch-and-bound search with total work limit of 30000 GWU, and for the external sampling method we use $L=30$ replications, each with a work limit of 1000 GWU. For external sampling, the 95\% confidence interval was computed by~\eqref{eq:saa_ci_ub}, and for internal sampling, the confidence 95\% confidence estimate for the upper bound for internal sampling is found by equation~\eqref{eq:internal_ci_ub}. 

In Table~\ref{tab:branching_compare_int_ext}, for the discrete distribution test instances we present the 95\% confidence upper bound on $\zpesp$ obtained by the internal sampling method and the external sampling method with three different sample sizes $N \in \{50,100,200\}$. We also present the 95\% confidence upper bound on $\zpesp$ obtained by the internal sampling on the continuous distribution instances, and for context, the table includes the point estimates of the perfect information bound and the best lower bound for each instance. 
Most of the best solutions were found by the greedy heuristic of Section~\ref{sec:heuristic}, and we provide more details of lower bound computations in Section~\ref{sec:comp-results-heuristic}. Note that the values of the PI Bound and the Best LB are point estimates of the values, while we present a 95\% confidence limit for the upper bounds obtained from branch-and-bound.  The standard error associated with the perfect information and lower bound estimates is in the range of $50-100$ and is given in Table~\ref{tab:bound-info-instances} in the appendix.


\begin{table}[hbtp]
\centering
\caption{95\% Confidence-Level Upper Bound Obtained by External and Internal Sampling Methods}
\label{tab:branching_compare_int_ext}
\begin{tabular}{c|rrr|r|rr}
\toprule
& \multicolumn{3}{|c|}{\bf External} & \multicolumn{1}{c|}{\bf Internal} & \multicolumn{1}{c}{\bf PI} & \multicolumn{1}{c}{\bf Best} \\
{\bf Instance} & {\bf $N=50$} & \multicolumn{1}{c}{\bf $N=100$} & \multicolumn{1}{c|}{\bf $N=200$} & \multicolumn{1}{c|}{\phantom{x}} & \multicolumn{1}{c}{\bf Bound} & \multicolumn{1}{c}{\bf LB}\\ \hline
J20\_1 & 9048.6 & 8881.5 & 8903.9 & 8834.1 & 9739.2 &  7919.0 \\
J20\_2 & 10807.6 & 10645.6 & 10683.3 & 10611.5 & 11465.6  & 9551.0 \\ 
J20\_3 & 12597.6 & 12378.1 & 12386.8 & 12283.4 & 13565.2 & 11454.0 \\
J25\_1 & 9702.3 & 9697.1 & 9744.2 & 9502.2 & 10404.7 & 8278.5 \\ 
J25\_2 & 11370.6 & 11330.7 & 11401.7 & 11172.5 & 12576.6 & 9756.4 \\ 
J25\_3 & 10668.6 & 10605.8 & 10671.2 & 10372.6 & 11389.8 & 8882.9 \\
J20\_1\_C & & & & 10401.5 & 11287.9 & 9278.4 \\
J20\_2\_C & & & & 10584.9 & 11631.4 & 9302.1 \\
J20\_3\_C & & & & 11279.8 & 12411.7 & 10015.2 \\
J25\_1\_C & & & & 11560.1 & 12330.1 & 9971.6 \\
J25\_2\_C & & & & 10064.7 & 10801.0 & 8787.1 \\
J25\_3\_C & & & & 9093.6 & 9796.7 & 7552.6 \\
\bottomrule
\end{tabular}
\end{table}

From Table~\ref{tab:branching_compare_int_ext} we observe that the internal sampling method gives moderately better upper bounds than the external sampling method within this fixed work limit on the discrete instances.  Second, comparing the change in upper bound obtained by the external sampling method as the sample size $N$ increases can provide an indication of the reduction in bias of the estimate of $\zpesp$ as $N$ increases.  If all externally-sampled instances were able to be solved to optimality within the 1000 GWU limit, we would expect the bounds to be monotonically-decreasing.  (Refer to Table~\ref{tab:branching_compare_saa} to see how many of the 30 instances were solved to optimality). Thus, we find that the upper bounds from $N=100$ are generally lower than those obtained with $N=50$, but this trend reverses when increasing $N$ to 200, due to the time limit being reached more often on the $N=200$ instances.  Finally, we find that the information-relaxation-based branch-and-bound method significantly improves over the upper bound obtained from the perfect-information relaxation.   We are unaware of other algorithms for this problem class that can yield similar improvement.

\subsection{Performance of Greedy Heuristic}
\label{sec:comp-results-heuristic}

Finally, we study the computational performance of the greedy heuristic method of Section~\ref{sec:heuristic} by comparing the value of the best solution found by the heuristic compared to other methods that more directly rely on solutions obtained by the branch-and-bound search.  

In our implementation of the greedy heuristic, at each iteration we solve $N_1=20$ two-stage stochastic programs, each having $N_3=50$ scenarios, to get our restricted set of solutions $Y$.  We use the same $N_1=20$ samples as the outer sample, and for each $\seta_S^k, k \in [N_1]$ in the sample, we take a conditional sample of size $N_2=20$ to estimate the value of $F(S)$ at step 3 of the heuristic given in Algorithm~\ref{alg:greedy}.  To estimate $F(S \cup \{j\})$ at step 5 of the algorithm, we use $K$-means clustering with $K=4$, as described in Section~\ref{sec:heuristic}.  


The value of each feasible solution, regardless of how it was generated, was estimated using the process described in Section~\ref{sss:slb}. Given a candidate probing set $S$, we take sample of size $N_1=25$ of $\rveta_S$, and for each $\eta_S^k$ in the sample we 
obtain a solution $\hat{y}^k, k \in [N_1]$, by solving a two-stage stochastic program \eqref{eq:ysel} with $N_3=100$ scenarios of $\rvxi$ generated conditionally on $\rveta_S = \eta_S^k$.
Then for each $k \in [N_1]$, an independent conditional sample of $\rvxi$ of size $N_2=2000$ is taken and the recourse problems $Q(\hat{y}^k,\hat{\xi}^{ki})$ are solved for $k \in [N_1]$ and $i \in [N_2]$ to obtain the lower bound via equation~\eqref{eq:lower_bound}.


Table~\ref{tab:sols_found} shows the estimated values of the best solution found by the greedy heuristic method, the best solution obtained from running the internal sampling branch-and-bound method,  and the best solution found during \emph{any} of the replications of the externally-sampled branch-and-bound trees.  As the values were all estimated using the sample sample sizes $N_1, N_2$ and $N_3$, all estimates have approximately the same standard error, which tends to be around 50. 

\begin{table}[hbtp]
\centering
\caption{Estimated Objective Values of Best Solutions Found}
\label{tab:sols_found}
\begin{tabular}{c|rrr}
\toprule
{\bf Instance} & {\bf Heuristic} & \multicolumn{1}{c}{\bf Internal} & \multicolumn{1}{c}{\bf External} \\ \hline
J20\_1 & 7859.8 & 7734.6 & 7808.4\\
J20\_2 & 9511.8 & 9484.9 & 9551.0\\
J20\_3 & 11353.6 & 11219.6 & 11280.1\\ \hline
J25\_1 & 8199.0 & 8212.8 & 8076.3\\
J25\_2 & 9654.7 & 9451.7 & 9638.4\\
J25\_3 & 8804.7 & 8751.3 & 8836.7\\ \hline
J20\_1\_C & 9278.3 & 9019.2 & -\\
J20\_2\_C & 9224.7 & 9136.7 & -\\
J20\_3\_C & 10015.2 & 9926.3 & -\\ \hline
J25\_1\_C & 9971.6 & 9807.7 & -\\
J25\_2\_C & 8732.4 & 8486.4 & -\\
J25\_3\_C & 7549.6 & 7455.0 & -\\
\bottomrule
\end{tabular}
\end{table}

When evaluating the results of Table~\ref{tab:sols_found}, it is important to take into consideration that the value is the \emph{best} observed from solutions in that category, and the different methods had different numbers of candidate solutions.  
The greedy heuristic method estimates solution value for the ten solutions whose initial estimates are best.  For the internal sampling approach, only the solution coming from the leaf node $P$ with the largest estimated lower bound value of $F([\xd] \setminus S^P_0)-\alpha([\xd] \setminus S^P_0)$ was evaluated.  For the externally-sampled approach, the best solution from each of the 30 replications for each of the branching methods and for each sample sizes $N \in \{50,100,200\}$ is re-evaluated.  (Thus, there are 240 potential solutions for the J20 instances and 180 potential solutions for the J25 instances). 


A primary takeaway from the comparison between solution methods is that the greedy heuristic performs quite well at finding high-quality solutions compared to other methods, especially when taking into account the difference in number of different solutions whose value was estimated.  This takeaway is reinforced by 
Figure \ref{fig:heur_j25} which shows histograms of the estimated objective values of the solutions for the 10 solutions evaluated by the heuristic and all the solutions found by external sampling, for two instances with discrete distribution and 25 customers.   

\begin{figure}[bthp]
    \begin{center}
        \includegraphics[width=.48\linewidth]{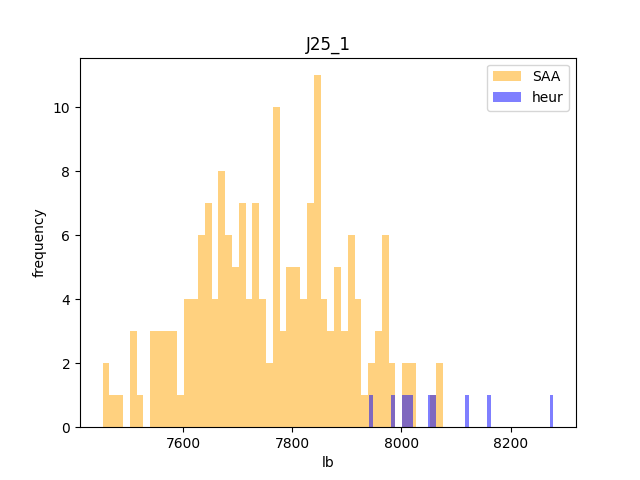}\quad\includegraphics[width=.48\linewidth]{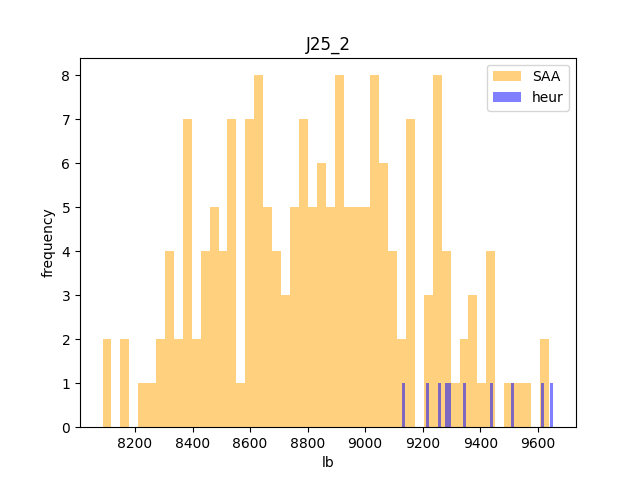}
        \caption{Distribution of Solution Values for J25\_1 and J25\_2}
        \label{fig:heur_j25}
    \end{center}
\end{figure}

The total computational effort for the heuristic is summarized in Table~\ref{tab:heur_effort}.  The table shows the total number of GWU required to perform the entire greedy heuristic (Heur), and the total GWU required to accurately estimate the value of all 10 of the perceived best solutions from the greedy procedure (Eval).  The table demonstrates that the computational effort of the heuristic is quite modest in comparison to the work limit imposed in the branch-and-bound methods.  The effort to perform a more accurate evaluation of the candidate solutions is more significant, but this effort is required, however, for obtaining an unbiased estimate of the objective value from a solution obtained by any method. 

\begin{table}[hbtp]
\centering
\caption{Work Units Required for the Heuristic}
\label{tab:heur_effort}
\begin{tabular}{c|rr}
\toprule
{\bf Instance} & \multicolumn{1}{c}{\bf Heur.} & \multicolumn{1}{c}{\bf Eval.} \\ \hline 
J20\_1 & 170.1 &  1735.7\\
J20\_2 & 279.1 &  5547.6\\
J20\_3 & 139.9 &  1847.4\\
J25\_1 & 373.7 &  4867.5\\
J25\_2 & 339.7 &  4998.5\\
J25\_3 & 521.4 &  9392.6\\
J20\_1\_C & 357.0 &  5589.5\\
J20\_2\_C & 466.0 &  6179.1\\
J20\_3\_C & 426.5 &  6235.2\\
J25\_1\_C & 856.8 &  12123.8\\
J25\_2\_C & 795.8 &  11245.9\\
J25\_3\_C & 1229.4 &  35910.2\\
\bottomrule
\end{tabular}
\end{table}




\section{Conclusions}

We have introduced the \emph{probing-enhanced stochastic program}, a new paradigm for modeling the decision-dependence of the distribution of random variables in stochastic programming.  We develop an information-relaxation-based branch-and-bound method for its solution that significantly outperforms the direct solution of a nonanticipative formulation of the problem and, when combined with our proposed sampling approximation, is the first method that can provide statistical bounds that improve upon perfect information bounds for decision-dependent stochastic programs having continuous distribution. Preliminary computational results indicate the promise of the approach. Interesting directions for future work include studying the potential to speed up the method by using decomposition methods to solve the two-stage stochastic programs that have to be solved to obtain bounds and testing the method on different information structures.



%
\section*{Conflict of interest}

The authors declare that they have no conflict of interest.


\section*{Appendix: Details of Branching Variable Selection}

Here we present the formulas used for computing $\hat{\Delta}_j^-$ which is used in the computation of the score for probing decision $j$ in Section \ref{sec:single-element-branching}. We describe this separately for the cases when internal sampling and external sampling is used.

\subsubsection*{External Sampling}

The first method for choosing $\hat{\Delta}_j^-$ requires estimating 
$\Expect_{\rveta_S} [R(\rveta_{S})|\rveta_j = 1]$ and $\Expect_{\rveta_S} [R(\rveta_{S})|\rveta_j = 0]$. Recall that when using external sampling, $F(S)$ is estimated via the formula
\[ \xfub(S) = 
{N}^{-1} \sum_{\seta_S \in \hat{\Eta}_S} |\Omega_N(\seta_S)| \xrub(\seta_S). \]
If we define $\hat{\Eta}_S^{jt} = \{ \seta_S \in \hat{\Eta}_S: \seta_j = t \}$, then we estimate
\[ \Expect_{\rveta_S} [R(\rveta_{S})|\rveta_j = t]
\approx (B^{jt})^{-1} \sum_{\seta_S \in \hat{\Eta}_S^{jt}} |\Omega_N(\seta_S)| \xrub(\seta_S)
\]
for $t=0,1$, where $B^{jt} = \sum_{\seta_S \in \hat{\Eta}_S^{jt}} |\Omega_N(\seta_S)|$. Note that this estimate uses all the same values $\xrub(\seta_S)$ that are already computed when computing the estimate $\xfub(S)$.

The second method for choosing $\hat{\Delta}_j^-$ requires estimating the covariance of $\rveta_j$ and $R(\rveta_S)$. This is computed via the estimate:
\[ {N}^{-1} \sum_{\seta_S \in \hat{\Eta}_S} |\Omega_N(\seta_S)| (\seta_j - m_j)(\xrub(\seta_S) - \xfub(S)), \]  
where $m_j = N^{-1} \sum_{\seta_S \in \hat{\Eta}_S} |\Omega_N(\seta_S)| \seta_j$.

\subsubsection*{Internal Sampling}

Recall that in this case $F(S)$ is estimated from \eqref{eq:fubdef} as
\[ \fub(S) = N_1^{-1} \sum_{k \in [N_1]} \rub(\eta_S^k). \]
If we define $\Omega^{tj} = \{ k \in [N_1]: \eta_j = t \}$, then we estimate
\[ \Expect_{\rveta_S} [R(\rveta_{S})|\rveta_j = t]
\approx |\Omega^{tj}|^{-1} \sum_{k \in \Omega^{tj}} \rub(\eta_S^k) 
\]
for $t=0,1$. Once again, this estimate uses all the same values $\rub(\seta_S^k)$ that are already computed when computing the estimate $\fub(S)$.

The covariance of $\rveta_j$ and $R(\rveta_S)$ is estimated via the formula
\begin{equation}
\label{eq:branching_covariance}
\hat{\Delta}_j^- = {N_1}^{-1} \sum_{k \in [N_1]} (\seta^k_j - m_j)(\rub(\seta_S^k) - \fub(S)),
\end{equation}
where $m_j = (N_1)_{-1} \sum_{k \in N_1} \seta^k_j$.

The estimates when the support of $\rveta$ or $\rvxi$ are small are adapted similarly.


Table~\ref{tab:bound-info-instances} shows an estimate of the perfect information bound (PI Bound) and the standard error of our estimate, as well as an estimate of the value of the best solution our methods found for the instance (Best LB), and the standard error of the estimate.  

\begin{table}[hbtp]
\caption{Bound Information for Instances in Computational Studies}
\label{tab:bound-info-instances}
\begin{center}
\begin{tabular}{c|rr|rr}
{\bf Name } & \multicolumn{1}{|c}{{\bf PI Bound}} & \multicolumn{1}{c|}{{\bf S.Err.}} & \multicolumn{1}{|c}{{\bf Best LB}} & \multicolumn{1}{c}{{\bf S.Err.}}\\ \hline 
J20\_1 & 9739.2 & 66.4 &  7919.0 & 59.5 \\
J20\_2 & 11465.6 & 72.1 & 9551.0 & 54.8 \\
J20\_3 & 13565.2 & 62.2 & 11454.0 & 61.3 \\
J25\_1 & 10404.7 & 68.3 & 8278.5 & 56.1 \\
J25\_2 & 12576.6 & 96.6 & 9756.4 & 48.1 \\
J25\_3 & 11389.8 & 74.2 & 8882.9 & 49.5 \\ \hline 
J20\_1\_C & 11287.9 & 62.4 & 9278.4 & 29.1 \\
J20\_2\_C & 11631.4 & 52.1 & 9302.1 & 56.4 \\
J20\_3\_C & 12411.7 & 65.0 & 10015.2 & 39.8 \\
J25\_1\_C & 12330.1 & 49.3 & 9971.6 & 20.6 \\
J25\_2\_C & 10801.0 & 59.9 & 8787.1 & 31.8\\
J25\_3\_C & 9796.7 & 49.3 & 7552.6 & 6.8\\
\end{tabular}
\end{center}
\end{table}

\end{document}